\documentclass[11pt,reqno]{amsart}
\usepackage{diagrams}
\usepackage{amssymb}
\usepackage{cite}
\usepackage[pagewise]{lineno}
\numberwithin{equation}{section}

\theoremstyle{definition}
\newtheorem{thm}{Theorem}[section]
\newtheorem{cor}[thm]{Corollary}

\newtheorem{prop}[thm]{Proposition}
\newtheorem{defi}[thm]{Definition}
\newtheorem{rem}[thm]{Remark}
\newtheorem{note}[thm]{Notation}

\DeclareMathOperator{\Hc}{\mathcal{H}om}

\DeclareMathOperator{\Ima}{\mathrm{Im}}

\DeclareMathOperator{\p3}{\mathbb{P}^3}

\DeclareMathOperator{\mo}{\mathcal{O}}

\newcommand{\mr}[1]{\mathrm{#1}}
\newcommand{\mb}[1]{\mathbb{#1}}
\newcommand{\mbf}[1]{\mathbf{#1}}
\newcommand{\mc}[1]{\mathcal{#1}}
\newcommand{\ov}[1]{\overline{#1}}

\newcommand{\mf}[1]{\mathfrak{#1}}

\begin{document}

\title[N\'{e}ron models of intermediate Jacobians]
{N\'{e}ron models of intermediate Jacobians associated to moduli spaces}

\author[A. Dan]{Ananyo Dan}

\address{BCAM - Basque Centre for Applied Mathematics, Alameda de Mazarredo 14,
48009 Bilbao, Spain}

\email{adan@bcamath.org}

\author[I. Kaur]{Inder Kaur}

\address{Pontifícia Universidade Cat\'{o}lica do Rio de Janeiro (PUC-Rio), R. Marquês de S\~{a}o Vicente, 225 - G\'{a}vea, Rio de Janeiro - RJ, 22451-900, Brasil}

\email{inder@impa.br}

\subjclass[2010]{Primary $14$C$30$, $14$C$34$, $14$D$07$, $32$G$20$, $32$S$35$, $14$D$20$, Secondary $14$H$40$}


\date{\today}

\begin{abstract}
Let $\pi_1:\mc{X} \to \Delta$ be a flat family of smooth, projective curves of genus $g \ge 2$, 
degenerating to an irreducible nodal curve $X_0$ with exactly one node.
Fix an invertible sheaf $\mc{L}$ on $\mc{X}$ of relative odd degree. 
Let $\pi_2:\mc{G}(2,\mc{L}) \to \Delta$ be the relative Gieseker moduli space of rank $2$ semi-stable vector bundles with determinant $\mc{L}$ over $\mc{X}$.
Since $\pi_2$ is smooth over $\Delta^*$, there exists a canonical family $\widetilde{\rho}_i:\mbf{J}^i_{\mc{G}(2, \mc{L})_{\Delta^*}} \to  \Delta^{*}$ of $i$-th intermediate Jacobians
i.e., for all $t \in \Delta^*$, $(\widetilde{\rho}_i)^{-1}(t)$ is the $i$-th intermediate Jacobian of $\pi_2^{-1}(t)$.
There exist different N\'{e}ron models $\ov{\rho}_i:\ov{\mbf{J}}_{\mc{G}(2, \mc{L})}^i \to \Delta$ extending $\widetilde{\rho}_i$ to the entire disc $\Delta$, constructed
by Clemens \cite{cle}, Saito \cite{sai}, Schnell \cite{schnel}, Zucker \cite{zuck} and Green-Griffiths-Kerr \cite{green}. 
In this article, we prove that in our setup, the N\'{e}ron model $\ov{\rho}_i$ is canonical in the sense that 
the different N\'{e}ron models coincide and is an analytic fiber space which graphs admissible normal functions. 
We also show that for 
$1 \le i \le \max\{2,g-1\}$, the central fiber of $\ov{\rho}_i$
is a fibration over product of copies of $J^k(\mr{Jac}(\widetilde{X}_0))$ for certain values of $k$,
where $\widetilde{X}_0$ is the normalization of $X_0$. 
In particular, for $g \ge 5$ and $i=2, 3, 4$, the central fiber of $\ov{\rho}_i$ is a semi-abelian variety.
Furthermore, we prove that the $i$-th generalized intermediate Jacobian of the (singular) central fibre of $\pi_2$ 
is a fibration over the central fibre of the N\'{e}ron model $\ov{\mbf{J}}^i_{\mc{G}(2, \mc{L})}$.
In fact, for $i=2$ the fibration is an isomorphism. 
 \end{abstract}

 \maketitle

\section{Introduction}
Throughout this article the underlying field will be $\mb{C}$.
 Given a smooth, projective variety $Y$, the $k$-th \emph{intermediate Jacobian of} $Y$, denoted $J^k(Y)$ is defined as:
 \begin{equation}\label{ner18}
  J^k(Y):=\frac{H^{2k-1}(Y,\mb{C})}{F^kH^{2k-1}(Y,\mb{C})+H^{2k-1}(Y,\mb{Z})},
 \end{equation}
 where $F^\bullet$ denotes the Hodge filtration.
 The intermediate Jacobian of a smooth, projective variety has been studied for decades and been used to investigate the  
 geometric and arithmetic properties of the variety (see for example \cite{cle, casa, JAD2}). 
 Using variation of Hodge structures \cite{grif1, grif2, GH}, one can further study families of intermediate Jacobians associated to
 smooth families of projective varieties (see for example \cite{grif4, bala1, kane, hori}). 
 In this article, we study the degeneration of certain families of intermediate Jacobians.
 
 Classically, degeneration of families of Jacobians of smooth, projective curves was studied using N\'{e}ron models (see \cite{bn}).
 This has been generalized to study degeneration of families of intermediate Jacobians of higher dimensional smooth, projective 
 varieties by Zucker \cite{zuck}, Clemens \cite{cle}, Saito \cite{sai} and more recently by Green-Griffiths-Kerr \cite{green} and Schnell \cite{schnel}.
 Unfortunately, not all the N\'{e}ron models mentioned in the literature are the same and in most cases are not Hausdorff. However, in the unipotent monodromy case, 
 the N\'{e}ron model of Green-Griffiths-Kerr (GGK) is more natural and arises as an (Hausdorff) analytic fiber space (see \cite{green, sai3}).
 The GGK-N\'{e}ron model has been generalized by Brosnan, Pearlstein and Saito in \cite{brosn2}.
 However, none of the existing literature describes the central fiber of any of the N\'{e}ron models mentioned above.
 In this article we  prove that the different N\'{e}ron models of families of intermediate Jacobians coincide in the case of families of 
 moduli spaces of rank $2$ semi-stable sheaves with fixed determinant.
 Note that Theorem \ref{intthm1} below is a vast generalization of \cite[Theorem $1.2$]{indpre}. In particular, 
 \cite[Theorem $1.2$]{indpre} is the special case of Theorem \ref{intthm1} when restricted to the 
 second intermediate Jacobian (see Corollary \ref{corsemiab}).
 Although this article as well as \cite{indpre} uses common tools from limit mixed Hodge structures, the main results of both articles are independent 
 from one another.
 The main purpose of this article is to  
 give a complete description of the central fiber of the N\'{e}ron model for \emph{all} families of intermediate Jacobians associated to 
 the relative moduli space and compare it with the generalized intermediate Jacobian of 
 the central fiber of the relative moduli space.
 
We fix notations.  Let $\pi_1:\mc{X} \to \Delta$ be a flat family of projective curves of genus $g \ge 2$, 
smooth over the punctured disc $\Delta^*$ such that the central fiber is an irreducible nodal curve $X_0$ with exactly one node.
Fix an invertible sheaf $\mc{L}$ on $\mc{X}$ of odd degree and let $\mc{L}_0:=\mc{L}|_{X_0}$. 
Denote by $\pi_2:\mc{G}(2,\mc{L}) \to \Delta$ the relative Gieseker moduli space of rank $2$ semi-stable sheaves on 
$\mc{X}$ with determinant $\mc{L}$ with central fiber, say $\mc{G}_{X_0}(2,\mc{L}_0)$ (see \S \ref{subsec1}).
   Recall, for every $t \in \Delta^*$, the fiber $\pi_2^{-1}(t)$ is isomorphic to the non-singular moduli space $M_{\mc{X}_t}(2,\mc{L}_t)$ 
   of rank $2$ semi-stable sheaves with determinant 
   $\mc{L}_t$ on $\mc{X}_t$, where $\mc{L}_t:=\mc{L}|_{\mc{X}_t}$  (see for example \cite{ind} for preliminaries on moduli spaces of sheaves with fixed determinant).
   Using the variation of Hodge structures, we obtain a family 
   \[\widetilde{\rho}: \mbf{J}^i_{\mc{G}(2,\mc{L})_{\Delta^*}} \to \Delta^*\] of $i$-th intermediate Jacobians
   such that for all $t \in \Delta^*$, we have $\widetilde{\rho}^{-1}(t)=J^i(M_{\mc{X}_t}(2,\mc{L}_t))$.
    By Theorem \ref{ntor08} below, there exists a GGK-N\'{e}ron model associated to $\widetilde{\rho}$:
   \[\ov{\rho}: \ov{\mbf{J}}^i_{\mc{G}(2,\mc{L})} \to \Delta.\]
   Note that $\ov{\rho}$ is an analytic fiber space and every holomorphic section of 
   $\widetilde{\rho}$ extends to a holomorphic section of $\ov{\rho}$. 
   We also show that the N\'{e}ron model $\ov{\mbf{J}}^i_{\mc{G}(2,\mc{L})}$ coincides with 
   the N\'{e}ron models of Clemens \cite{cle} and Saito \cite{sai} (see Corollary \ref{tor02}).
   We then prove:
  \begin{thm}\label{intthm1} 
  For any $1 \le i \le \max\{2, g-1\}$, the central fiber $\left(\ov{\mbf{J}}^i_{\mc{G}(2, \mc{L})}\right)_0$ of the N\'{e}ron model is a 
  fibration over $\prod\limits_{k=1}^{[\frac{g}{2}]} J^k(\mr{Jac}(\widetilde{X}_0))^{d_{i,k}}$ with every fiber isomorphic to 
  \[\frac{H^{2i-4}(M_{\widetilde{X}_0}(2,\widetilde{\mc{L}}_0),\mb{C})}{F^{i-1}H^{2i-4}(M_{\widetilde{X}_0}(2,\widetilde{\mc{L}}_0),\mb{C})+H^{2i-4}(M_{\widetilde{X}_0}(2,\widetilde{\mc{L}}_0),\mb{Z})}\]
  where $\widetilde{X}_0$ is the normalization of $X_0$, $\widetilde{\mc{L}}_0$ is the pull-back of $\mc{L}_0$ to $\widetilde{X}_0$ and 
  $d_{i,k}$ is the coefficient of $t^{i-3k+1}$ of the polynomial
 $(1+t^3)(1+t+t^2+...+t^{g-1-2k})(1+t^2+t^4+...+t^{2(g-1)-4k}).$
\end{thm}
See Theorem \ref{ntor12} for a more general statement and proof. Theorem  \ref{intthm1} implies that the central fiber of the N\'{e}ron model
is never an abelian variety. However, we observe that for $i=2, 3$ and $4$, the central fiber of the N\'{e}ron model is a 
semi-abelian variety (Corollary \ref{corsemiab}). Recall, the $i$-th generalized intermediate Jacobian of $\mc{G}_{X_0}(2,\mc{L}_0)$, 
denoted $J^i(\mc{G}_{X_0}(2,\mc{L}_0))$, is defined analogously to \eqref{ner18}, with the relevant cohomology groups equipped with a mixed Hodge structure.
We prove:
\begin{thm}\label{intthm2}
The $i$th-intermediate Jacobian $J^i(\mc{G}_{X_0}(2,\mc{L}_0))$ is a fibration over $\left(\ov{\mbf{J}}^i_{\mc{G}(2, \mc{L})}\right)_0$ with every fiber isomorphic to
$J^{i-1}(M_{\widetilde{X}_0}(2, \widetilde{\mc{L}}_0)) \times J^{i-2}(M_{\widetilde{X}_0}(2, \widetilde{\mc{L}}_0)) \times J^{i-3}(M_{\widetilde{X}_0}(2, \widetilde{\mc{L}}_0))$.
\end{thm}
See Theorem \ref{ner16} for a proof. This answers a question posed by Green, Griffiths and Kerr in \cite[p. $293$]{green} for $\mc{G}(2,\mc{L})$.
As a consequence, we observe that $J^2(\mc{G}_{X_0}(2,\mc{L}_0))$ is isomorphic to the central fiber 
$\left(\ov{\mbf{J}}^2_{\mc{G}(2, \mc{L})}\right)_0$ of the N\'{e}ron model. In particular, $J^2(\mc{G}_{X_0}(2,\mc{L}_0))$ is a semi-abelian variety (see Corollary \ref{ntor18}).
Theorems \ref{intthm1} and \ref{intthm2} generalize the classical result that
$\mr{Jac}(X_0)$ is a $\mathbb{C}^*$-fibration over $\mr{Jac}(\widetilde{X}_0)$. We now discuss the strategy of the proofs.

Denote by $\mc{G}(2,\mc{L})_\infty$ (resp. $\mc{X}_\infty$) the base change of $\mc{G}(2,\mc{L})$ (resp. $\mc{X}$)
under the composed morphism $\mf{h} \xrightarrow{e} \Delta^* \hookrightarrow \Delta$, where $\mf{h}$ is the universal cover of $\Delta^*$.
The cohomology groups $H^i(\mc{G}(2,\mc{L})_\infty, \mb{Z})$ and $H^i(\mc{X}_\infty, \mb{Z})$ are equipped with a (limit) mixed Hodge structure (see Theorem \ref{tor25}).
Moreover, one has a natural monodromy action on $H^i(\mc{G}(2,\mc{L})_\infty, \mb{Z})$. Denote by $N_{i,\mb{C}}$ (resp. $N_{i,\mb{Z}}$)
the monodromy invariant subspace of $H^i(\mc{G}(2,\mc{L})_\infty, \mb{C})$ (resp. $H^i(\mc{G}(2,\mc{L})_\infty, \mb{Z})$). Note that the mixed Hodge structure
on $H^i(\mc{G}(2,\mc{L})_\infty, \mb{Z})$ induces a mixed Hodge structure on $N_{i,\mb{Z}}$ (see \cite[Chapter $11$]{pet}). Let 
$J'_i:=N_{2i-1,\mb{C}}/(F^i N_{2i-1,\mb{C}}+N_{2i-1,\mb{Z}})$.
The central fiber $\left(\ov{\mbf{J}}^i_{\mc{G}(2, \mc{L})}\right)_0$ of the GGK-N\'{e}ron model sits in the short exact sequence:
\[0 \to J'_i \to \left(\ov{\mbf{J}}^i_{\mc{G}(2, \mc{L})}\right)_0 \to G_{2i-1} \to 0,\]
where $G_{2i-1}$ is a finite group encoding the monodromy action on $H^{2i-1}(\mc{G}(2,\mc{L})_\infty, \mb{Z})$ (Theorem \ref{ntor08}).
We prove that in our setup, the group $G_{2i-1}$ vanishes (Theorem \ref{ntor14}). As a consequence, 
$\left(\ov{\mbf{J}}^i_{\mc{G}(2, \mc{L})}\right)_0$ is connected. 

Recall, the $i$-th intermediate Jacobian of a smooth, projective variety $Y$ is a quotient of $H^{2i-1}(Y,\mb{C})$, which is a pure Hodge structure
of weight $2i-1$. Now that we have $J'_i \cong \left(\ov{\mbf{J}}^i_{\mc{G}(2, \mc{L})}\right)_0$ and 
$J'_i$ is a quotient of $N_{2i-1,\mb{C}}$ (which is a mixed Hodge structure), it is natural to ask if the image of 
$\mr{Gr}^W_{2i-1} N_{2i-1,\mb{C}}$ in $\left(\ov{\mbf{J}}^i_{\mc{G}(2, \mc{L})}\right)_0$ (see Definition \ref{def:ner01}), 
which we denote by $pure\left(\ov{\mbf{J}}^i_{\mc{G}(2, \mc{L})}\right)_0$,
is an abelian variety. We prove that (see Theorem \ref{ntor12} and Corollary \ref{corsemiab}):
\begin{thm}\label{intthm3}
 For $1 \le i \le \max\{2, g-1\}$, $pure\left(\ov{\mbf{J}}^i_{\mc{G}(2, \mc{L})}\right)_0 \cong 
 \prod\limits_{k=1}^{[\frac{g}{2}]} J^k(\mr{Jac}(\widetilde{X}_0))^{d_{i,k}}$, where $d_{i,k}$ is as in Theorem \ref{intthm1}.
 In particular, for $g \ge 5$ and $i=2,3, 4$, $pure\left(\ov{\mbf{J}}^i_{\mc{G}(2, \mc{L})}\right)_0$ is an abelian variety.
\end{thm}
One of the key steps to prove this theorem is to show that there 
exists an isomorphism of mixed Hodge structures from $H^1(\mc{X}_\infty,\mb{Z})$ to $H^3(\mc{G}(2,\mc{L})_\infty,\mb{Z})$ (Theorem \ref{ner01}).
This is a generalization to the relative setup of a classical result \cite[Proposition $1$]{mumn} of Mumford and Newstead. Finally, we 
compute the kernel of the natural morphism from $\left(\ov{\mbf{J}}^i_{\mc{G}(2, \mc{L})}\right)_0$ to $pure\left(\ov{\mbf{J}}^i_{\mc{G}(2, \mc{L})}\right)_0$.
This will give us a complete description of the central fiber of the N\'{e}ron model as given in Theorem \ref{intthm1}.

We remark that Theorems \ref{intthm1} and \ref{intthm3} still hold if we replace $\mc{G}(2,\mc{L})$ by the (relative) Simpson's 
moduli space of rank $2$ semi-stable sheaves with determinant $\mc{L}$ as defined in \cite{sun1}.
This is because both (relative) moduli spaces coincide over $\Delta^*$, hence have the same N\'{e}ron models of the associated family of 
intermediate Jacobians. 
We use (relative) Gieseker moduli space simply because
the central fiber of this moduli space is a simple normal crossings divisor, which makes computations using Steenbrink spectral sequence possible.

{\emph{Applications and further questions:}} 
Using Theorems \ref{ner01} and \ref{ntor14}, one can prove the higher rank Torelli theorem for $\mc{G}_{X_0}(2,\mc{L}_0)$ (see \cite{indpre}). 
This is a generalization to the nodal curve case of a classical result of Mumford and Newstead \cite{mumn}.
Since the above N\'{e}ron models graph admissible normal functions (i.e., holomorphic sections of $\widetilde{\rho}$ extend holomorphically to that of $\ov{\rho}$), 
another application is to study the limit Abel-Jacobi map as described by Green, Griffiths and Kerr in \cite{green}.

Compactification of Jacobians of curves and moduli spaces is an active topic of research in algebraic geometry 
 (see for example \cite{alex2, capon, est}). Analogously one can ask, what is
the compactification of the $i$-th intermediate Jacobian 
$J^i(\mc{G}_{X_0}(2,\mc{L}_0))$?
 By Theorem \ref{intthm2}, 
the fibers to the natural morphism from $J^i(\mc{G}_{X_0}(2,\mc{L}_0))$ to $\left(\ov{\mbf{J}}^i_{\mc{G}(2, \mc{L})}\right)_0$ are abelian varieties.
Therefore, by Theorem \ref{intthm1}, to compactify $J^i(\mc{G}_{X_0}(2,\mc{L}_0))$ we simply need to obtain a suitable compactification of  
$H^{2i-4}(M_{\widetilde{X}_0}(2,\widetilde{\mc{L}}_0),\mb{C})/(F^{i-1}H^{2i-4}(M_{\widetilde{X}_0}(2,\widetilde{\mc{L}}_0),\mb{C})+H^{2i-4}(M_{\widetilde{X}_0}(2,\widetilde{\mc{L}}_0),\mb{Z}))$
which deforms ``uniformly'' along $\prod\limits_{k=1}^{[\frac{g}{2}]} J^k(\mr{Jac}(\widetilde{X}_0))^{d_{i,k}}$. 
We pursue this question in future work.

 \emph{Outline}: In \S \ref{nsec1} we review preliminaries on N\'{e}ron models
 and limit mixed Hodge structure. 
 In \S \ref{nsec3} we study the monodromy action on the relative Gieseker moduli space and show that the different N\'{e}ron models coincide. 
 In \S \ref{nsec4} we give a geometric description of the central
 fiber of the GGK-N\'{e}ron model. In \S \ref{nsec5}, we introduce the $i$-th intermediate Jacobian of $\mc{G}_{X_0}(2,\mc{L}_0)$ and relate it to the 
 central  fiber of the N\'{e}ron model.

\vspace{0.2 cm}
\emph{Acknowledgements} 
We thank Prof. J. F. de Bobadilla, Dr. B. Sigurdsson and Dr. S. Basu for numerous discussions. 
The first author is currently supported by ERCEA Consolidator Grant $615655$-NMST and also
by the Basque Government through the BERC $2014-2017$ program and by Spanish
Ministry of Economy and Competitiveness MINECO: BCAM Severo Ochoa
excellence accreditation SEV-$2013-0323$. 
The second author is funded by CAPES-PNPD scholarship. 

\section*{List of Notations}

\begin{center}
 \begin{tabular}{l p{10cm}}
  $X_0,x_0$ & irreducible nodal curve $X_0$ with node at $x_0$\\
  $\pi: \widetilde{X}_0 \to X_0$ & normalization of $X_0$\\
  $\Delta, \Delta^*$ & open, unit disc $\Delta$ and $\Delta^*:=\Delta \backslash \{0\}$\\
  $\rho: \mc{Y} \to \Delta$ & family of projective varieties, smooth over $\Delta^*$\\
  $\mc{Y}_t$ & the fiber $\rho^{-1}(t)$ for any $t \in \Delta$\\
  $\mc{Y}_\infty$ & the base change of the family $\rho$ under the natural morphism $\mf{h} \to \Delta^* \hookrightarrow \Delta$, where
  $\mf{h}$ is the universal covering of $\Delta^*$\\
  $\mc{Y}_{\Delta^*}$ & restriction of $\mc{Y}$ to $\Delta^*$\\
  $\mc{H}^i_{\mc{Y}_{\Delta^*}}, F^p\mc{H}^i_{\mc{Y}_{\Delta^*}}$ & Hodge bundles associated to the family $\mc{Y}_{\Delta^*}$\\
  $\ov{\mc{H}}^i_{\mc{Y}_{\Delta^*}}, F^p\ov{\mc{H}}^i_{\mc{Y}_{\Delta^*}}$ & canonical extensions of 
  $\mc{H}^i_{\mc{Y}_{\Delta^*}}, F^p\mc{H}^i_{\mc{Y}_{\Delta^*}}$, respectively\\
  $\widetilde{\rho}: \mbf{J}^i_{\mc{Y}_{\Delta^*}} \to \Delta^*$ & family of $i$-th intermediate Jacobians associated to  $\mc{Y}_{\Delta^*}$\\
  $\ov{\rho}: \ov{\mbf{J}}^i_{\mc{Y}} \to \Delta$ & N\'{e}ron model associated to $\widetilde{\rho}$\\
  $T_{s,i}, T_{s,i}^{\mb{Q}}$ & local monodromy transformation associated to  $\rho$\\
  $T_i:H^i(\mc{Y}_\infty, \mb{Q}) \to H^i(\mc{Y}_\infty, \mb{Q})$ & limit monodromy transformation\\
  $N_i$ & $\log(T_i)$\\
  $\mr{sp}_i: H^i(\mc{Y}_0,\mb{Z}) \to H^i(\mc{Y}_\infty, \mb{Z})$ & specialization morphism\\
  $M_Y(2,\mc{L}')$ & moduli space of rank $2$, semi-stable sheaves with determinant $\mc{L}'$ over $Y$\\
  $\pi_1: \mc{X} \to \Delta$ & family of projective curves with central fiber $X_0$, smooth over $\Delta^*$\\
  $\mc{L}, \mc{L}_0, \widetilde{\mc{L}}_0$ & odd degree invertible sheaf $\mc{L}$ on $\mc{X}$, $\mc{L}_0:=\mc{L}|_{X_0}$, $\widetilde{\mc{L}}_0:=\pi^*\mc{L}_0$\\
  $\widetilde{\pi}_1: \widetilde{\mc{X}} \to \mc{X} \xrightarrow{\pi} \Delta$ & blow-up of $\mc{X}$ at $x_0$\\
  $\pi_2: \mc{G}(2,\mc{L}) \to \Delta$ & relative Gieseker moduli space associated to $\pi_1$\\
  $\mc{G}_{X_0}(2,\mc{L}_0)$ & central fiber of the moduli space $\mc{G}(2,\mc{L})$\\
  $\mc{G}_0, \mc{G}_1$ & the two irreducible components of  $\mc{G}_{X_0}(2,\mc{L}_0)$
 \end{tabular}
\end{center}

\section{Preliminaries: N\'{e}ron models of families of intermediate Jacobians}\label{nsec1}
In this section, we recall preliminaries on the N\'{e}ron model of families of intermediate Jacobians. 
We assume basic familiarity with limit mixed Hodge structures. See \cite[\S $11$]{pet} for a detailed study. 

\begin{note}
 Let $\rho:\mc{Y} \to \Delta$ be a flat family of projective varieties, smooth over $\Delta^*$. Let 
 $\rho':\mc{Y}_{\Delta^*} \to \Delta^*$ be the restriction 
    of $\rho$ to $\Delta^*$. 
\end{note}

\subsection{Families of intermediate Jacobians}
Denote by $\mb{H}_{\mc{Y}_{\Delta^*}}^i:=R^i{\rho}'_{*}\mb{Z}$.
By Ehresmann's theorem (see \cite[Theorem $9.3$]{v4}), we have for all $i \ge 0$, 
$H^i(\mc{Y}_t, \mb{Z})$ is constant as $t$ varies over all $t \in \Delta^*$. This implies that $\mb{H}_{\mc{Y}_{\Delta^*}}^i$ is a local system.
The associated vector bundle \[\mc{H}_{\mc{Y}_{\Delta^*}}^i:=\mb{H}_{\mc{Y}_{\Delta^*}}^i \otimes_{\mb{Z}} \mo_{\Delta^*}\]
 is called the \emph{Hodge bundle}. 
 There exist sub-bundles $F^p\mc{H}_{\mc{Y}_{\Delta^*}}^i \subset \mc{H}_{\mc{Y}_{\Delta^*}}^i$
defined by the condition: for any $t \in \Delta^*$, the fibers \[\left(F^p\mc{H}_{\mc{Y}_{\Delta^*}}^i\right)_t \subset \left(\mc{H}_{\mc{Y}_{\Delta^*}}^i\right)_t\]
can be identified respectively with 
$F^pH^i(\mc{Y}_t,\mb{C}) \subset H^i(\mc{Y}_t,\mb{C})$,
 where $F^p$ denotes the Hodge filtration (see \cite[\S $10.2.1$]{v4}).
 Using the Hodge bundle $\mc{H}_{\mc{Y}_{\Delta^*}}^{2i-1}$ and the sub-bundle $F^i\mc{H}_{\mc{Y}_{\Delta^*}}^{2i-1}$
 one can show that there exists a holomorphic family of principally polarized abelian varieties 
 \begin{equation}\label{ntor04}
  \widetilde{\rho}:\mbf{J}^i_{\Delta^*} \to \Delta^*
 \end{equation}
 such that  $\widetilde{\rho}^{-1}(s)=J^i(\mc{Y}_s)$ for every $s \in \Delta^*$.
 
 \subsection{Limit mixed Hodge structures}
 Consider the universal cover $\mf{h} \to \Delta^*$ of the punctured unit disc. 
 Denote by $e:\mf{h} \to \Delta^* \xrightarrow{j} \Delta$ the composed morphism and  
 $\mc{Y}_\infty:=\mc{Y} \times_{\Delta} \mf{h}$ the base change of the family $\mc{Y}$ over $\Delta$ to $\mf{h}$, by the morphism $e$.
 There exists an unique  \emph{canonical extension} $\ov{\mc{H}}_{\mc{Y}}^i$, extending  
 ${\mc{H}}_{\mc{Y}_{\Delta^*}}^i$ to the entire disc $\Delta$ (see \cite[Definition $11.4$]{pet} for the precise definition of 
 canonical extension).
 One can observe that $\ov{\mc{H}}_{\mc{Y}}^i$ is locally-free over $\Delta$.
 There is an explicit identification of the central fiber of the canonical extension $\ov{\mc{H}}_{\mc{Y}}^i$ 
 and the cohomology group $H^i(\mc{Y}_{\infty},\mb{C})$, depending on the choice of the parameter $t$ on $\Delta$ (see \cite[XI-$8$]{pet}):
 \begin{equation}\label{tor23}
  g^i_{_t}:H^i(\mc{Y}_{\infty},\mb{C}) \xrightarrow{\sim} \left(\ov{\mc{H}}_{\mc{Y}}^i\right)_0.
 \end{equation}
 Denote by $j:\Delta^* \to \Delta$
 the inclusion morphism.
 Note that
 $F^p\ov{\mc{H}}_{\mc{Y}}^i:= j_*\left(F^p\mc{H}_{\mc{Y}_{\Delta^*}}^i\right) \cap  \ov{\mc{H}}_{\mc{Y}}^i$
  is the \emph{unique largest} locally-free
 sub-sheaf of $\ov{\mc{H}}_{\mc{Y}}^i$ which extends $F^p\mc{H}_{\mc{Y}_{\Delta^*}}^i$.
 Denote by \[F^pH^i(\mc{Y}_{\infty},\mb{C}):=(g_{_t}^i)^{-1}\left(F^p\ov{\mc{H}}_{\mc{Y}}^i\right)_0.\] 
 Note that $F^\bullet$ does not always induce a pure Hodge structure 
 on $H^i(\mc{Y}_{\infty},\mb{C})$. However, we will observe that there is a mixed 
 Hodge structure on $H^i(\mc{Y}_{\infty},\mb{C})$ with good specialization properties.
 For this purpose, we first recall the monodromy transformation.
 
 For the rest of the section we assume that the central fiber of the family $\rho$ is a 
 reduced simple normal crossings divisor.
 For any $s \in \Delta^*$ and $i \ge 0$, denote by 
\[T_{s,i}: H^i(\mc{Y}_s,\mb{Z}) \to H^i(\mc{Y}_s,\mb{Z}) \, \mbox{ and }\, T_{s,i}^{\mb{Q}}: H^i(\mc{Y}_s,\mb{Q}) \to H^i(\mc{Y}_s,\mb{Q})\]
the \emph{local monodromy transformations}
   associated to the local system $\mb{H}_{\mc{Y}_{\Delta^*}}^i$ defined by parallel transport along a counterclockwise loop about $0 \in \Delta$
   (see \cite[\S $11.1.1$]{pet} or \cite[\S $3.1.1$]{v5}).
 By \cite[Theorem II.$1.17$]{deli2}  (see also \cite[Proposition I.$7.8.1$]{kuli}) the automorphism extends to a $\mb{Q}$-automorphism 
\begin{equation}\label{int01}
 T_i: H^i(\mc{Y}_{\infty},\mb{Q}) \to H^i(\mc{Y}_{\infty},\mb{Q}).
\end{equation}
Denote by $T_{i,\mb{C}}$ the induced automorphism on $H^i(\mc{Y}_{\infty},\mb{C})$. Denote by $N_{i,\mb{C}}:=\log(T_{i, \mb{C}})$. 
We now recall the following useful result in limit mixed Hodge structures:

 \begin{thm}\label{tor25}
  There exists an unique increasing \emph{monodromy weight filtration} $W_\bullet$ on $H^i(\mc{Y}_\infty,\mb{Q})$ such that
 \begin{enumerate}
  \item  for $i \ge 2$, $N_i(W_jH^i(\mc{Y}_\infty,\mb{Q})) \subset W_{j-2}H^i(\mc{Y}_\infty,\mb{Q})$, where $N_i:=\log(T_i)$ for 
  $T_i$ as in \eqref{int01},
  \item the map $N_i^l: \mr{Gr}^W_{i+l} H^i(\mc{Y}_\infty,\mb{Q}) \to \mr{Gr}^W_{i-l} H^i(\mc{Y}_\infty,\mb{Q})$ 
  is an isomorphism for all $l \ge 0$.
   \end{enumerate}
  The triple $(H^i(\mc{Y}_{\infty},\mb{Z}),W_\bullet,F^\bullet)$ then defines a mixed Hodge structure on $H^i(\mc{Y}_{\infty},\mb{Z})$, called the \emph{limit mixed Hodge structure}.
 \end{thm}
 
\begin{proof}
 See \cite[Lemma-Definition $11.9$]{pet} and \cite[Theorem $6.16$]{schvar}.
\end{proof}

Recall, for any $s \in \Delta^*$, there is a natural specialization morphism 
from $\mc{Y}_s$ to the central fiber $\mc{Y}_0$ of $\mc{Y}$. This induces a natural 
morphism from $H^i(\mc{Y}_0, \mb{Z})$ to $H^i(\mc{Y}_s, \mb{Z})$, which is not a morphism 
of Hodge structures. However, after identifying $H^i(\mc{Y}_s, \mb{Z})$ with 
$H^i(\mc{Y}_{\infty},\mb{Z})$, the resulting \emph{specialization morphism}
\[\mr{sp}_i: H^i(\mc{Y}_0, \mb{Z}) \to H^i(\mc{Y}_{\infty},\mb{Z})\]
is a morphism of mixed Hodge structures, with the limit mixed Hodge structure on 
$H^i(\mc{Y}_{\infty},\mb{Z})$ and the mixed Hodge structure on $H^i(\mc{Y}_0, \mb{Z})$
as defined in \cite[Example $3.5$]{ste1}.
By the local invariant cycle theorem \cite[Theorem $11.43$]{pet}, we have the following exact sequence of mixed Hodge structure:
 \begin{equation}\label{t02}
  H^i(\mc{Y}_0,\mb{Q}) \xrightarrow{\mr{sp}_i} H^i(\mc{Y}_{\infty},\mb{Q}) \xrightarrow{N_i/(2\pi \sqrt{-1})} H^i(\mc{Y}_{\infty},\mb{Q})(-1).
 \end{equation}
 
 We now recall the following useful computation of limit mixed Hodge structures:
\begin{prop}\label{ntor03}
  Suppose that the central fiber $\mc{Y}_0$ is a reduced, simple normal crossings divisor 
  consisting of two smooth, irreducible components, say $Y_1, Y_2$. 
  Then, we have the following exact sequence of mixed Hodge structures:
  \begin{equation}\label{ntor02}
H^{i-2}(Y_1 \cap Y_2, \mb{Q})(-1) \xrightarrow{f_i} H^i(\mc{Y}_0,\mb{Q}) \xrightarrow{\mr{sp}_i} H^i(\mc{Y}_\infty, \mb{Q}) \xrightarrow{g_i} \mr{Gr}^W_{i+1} H^i(\mc{Y}_\infty, \mb{Q}) \to 0,
  \end{equation}
where $f_i$ comes from the natural Gysin morphism, $\mr{sp}_i$ is the specialization morphism 
and $g_i$ is the natural projection.
 \end{prop}
 
 \begin{proof}
  See \cite[Corollary $2.4$]{mumf} for a proof.
 \end{proof}
 
 One of the important applications of Proposition \ref{ntor03} is the following limit mixed Hodge structure computation 
 associated to a degenerating family of curves. 

 \begin{thm}\label{ntor11}
 Let $g \ge 2$ be an integer, $\rho:\widetilde{\mc{X}} \to \Delta$ be a flat family of projective curves with $\widetilde{\mc{X}}$ regular, 
  $\widetilde{\mc{X}}_t$ is smooth of genus $g$ for all $t \in \Delta^*$ and central fiber 
  $\widetilde{\mc{X}}_0=Y_1 \cup Y_2$ with $Y_1 \cong \mb{P}^1$, 
  $Y_2$ smooth, irreducible and intersecting $Y_1$ transversally at two points, say $y_1, y_2$. 
 Then, there exists a basis $e_1, e_2, ..., e_{2g}$ of 
  $H^1(\widetilde{\mc{X}}_\infty, \mb{Z})$ such that 
  \begin{enumerate}
   \item $e_g$ (resp. $e_{2g}$) generates $\mr{Gr}^W_0H^1(\widetilde{\mc{X}}_\infty, \mb{Q})$ (resp. $\mr{Gr}^W_2H^1(\widetilde{\mc{X}}_\infty, \mb{Q})$),
   \item $e_1, e_2, ..., e_{g-1}, e_{g+1}, e_{g+2},..., e_{2g-1}$ form a basis of $\mr{Gr}^W_1H^1(\widetilde{\mc{X}}_\infty, \mb{Q})$,
   \item $e_i \cup e_{j} \not= 0$ if and only if $|j-i|= g$.
  \end{enumerate}
 \end{thm}

 \begin{proof}
  See \cite[Theorem $2.5$]{mumf} for a proof.
 \end{proof}

\subsection{Existence of N\'{e}ron models}
Note that $\ker N_{i, \mb{C}}$ is a sub-Hodge structure of $H^i(\mc{Y}_\infty, \mb{Q})$. Suppose that 
$H^i(\mc{Y}_s,\mb{Z})$ is torsion-free. Denote by 
\begin{equation}\label{ntor05}
 F^n\ker N_{i, \mb{C}}:= \ker N_{i, \mb{C}} \cap F^n(H^i(\mc{Y}_{\infty},\mb{C})) \mbox{ and } 
 G_{s,i}:=\frac{((T_{s,i}^{\mb{Q}}-\mr{Id})H^i(\mc{Y}_s,\mb{Q})) \cap H^i(\mc{Y}_s,\mb{Z})}{(T_{s,i}-\mr{Id})H^i(\mc{Y}_s,\mb{Z})}.
\end{equation}
 Note that as a group, $G_{s,i}$ does not depend on the choice of $s \in \Delta^*$, so we will denote this by $G_i$.
 Using the explicit description of $g^i_{_t}$ as in \cite[XI-$6$]{pet}, one can check that 
 \[\ker(T_i- \mr{Id}) \cap H^i(\mc{Y}_{\infty},\mb{Z}) \xrightarrow[\sim]{g^i_{_t}} \left(\ov{\mb{H}}^i_{\mc{Y}} \right)_0 \mbox{ where } 
  \ov{\mb{H}}^i_{\mc{Y}}:=j_*\mb{H}^i_{\mc{Y}_{\Delta^*}} \mbox{ and } \left(\ov{\mb{H}}^i_{\mc{Y}} \right)_0:=\ov{\mb{H}}^i_{\mc{Y}} \otimes k(o).\]
 Since $\ker(N_i)=\ker(T_i-\mr{Id})$, this implies  $\left(\ov{\mb{H}}^i_{\mc{Y}} \right)_0 \subset \ker(N_{i, \mb{C}})$. Denote by 
 \begin{equation}\label{ntor13}
 J'_m:=\frac{\ker N_{2m-1, \mb{C}}}{F^m \ker N_{2m-1}+\left(\ov{\mb{H}}^{2m-1}_{\mc{Y}} \right)_0}=\frac{\ker N_{2m-1, \mb{C}}}{F^m \ker N_{2m-1}+(\ker(T_{2m-1}- \mr{Id}) \cap H^{2m-1}(\mc{Y}_{\infty},\mb{Z}))}.
 \end{equation}
  There exists a N\'{e}ron model associated to the family of intermediate Jacobians $\mbf{J}^i_{\Delta^*}$ in the following sense:
 
 \begin{thm}\label{ntor08}
  There exists a canonical analytic fiber space, called the \emph{N\'{e}ron model} of $\widetilde{\rho}$,
  \[\ov{\rho}:\ov{\mbf{J}}^i_{\mc{Y}} \to \Delta\]
  extending $\widetilde{\rho}$ such that every holomorphic section of $\widetilde{\rho}$ extends to a holomorphic section of $\ov{\rho}$. In particular,
  for all $s \in \Delta^*$, the fiber $\ov{\rho}^{-1}(s)=\widetilde{\rho}^{-1}(s)=J^i(\mc{Y}_s)$
  and the central fiber of $\ov{\rho}$ sits in the following short exact sequence:
  \begin{equation}\label{ntor09}
   0 \to J'_i \to \left(\ov{\mbf{J}}^i_{\mc{Y}} \right)_0 \to G_{2i-1} \to 0.\
  \end{equation}
 \end{thm}

 \begin{proof}
  See \cite[Theorem II.B.$9$]{green} for proof of the statement.
 \end{proof}
 
 As mentioned in the introduction, it is natural to study the geometry of the ``pure weight part'' of the central fiber of 
 the N\'{e}ron model. We define this below:
 \begin{defi}\label{def:ner01}
  Since $N_{i, \mb{C}}$ is a morphism of mixed Hodge structures, $\ker N_{i, \mb{C}}$ is equipped with a natural mixed Hodge structure.
  The \emph{pure weight part of } $\left(\ov{\mbf{J}}^i_{\mc{Y}} \right)_0$ will be defined as 
  \[{pure}\left(\left(\ov{\mbf{J}}^i_{\mc{Y}} \right)_0\right):= \mr{coker}\left(W_{2i-2} \ker N_{2i-1, \mb{C}} \hookrightarrow
   \ker N_{2i-1, \mb{C}} \twoheadrightarrow J'_i \to \left(\ov{\mbf{J}}^i_{\mc{Y}} \right)_0\right).
  \]
   The reason for this terminology is that one can observe that ${pure}\left(\left(\ov{\mbf{J}}^i_{\mc{Y}} \right)_0\right)$ is isomorphic to the image 
 $\mr{Gr}^W_{2i-1} \ker N_{2i-1, \mb{C}}$ under the composition,
 \[\mr{Gr}^W_{2i-1} \ker N_{2i-1, \mb{C}} \hookrightarrow
   \ker N_{2i-1, \mb{C}} \twoheadrightarrow J'_i \to \left(\ov{\mbf{J}}^i_{\mc{Y}} \right)_0.\]
 \end{defi}

 \section{Monodromy action on local systems associated to moduli spaces}\label{nsec3}

  In this section we study the monodromy action on the local systems associated to families 
 of moduli spaces of semi-stable sheaves on projective curves. In the next section, we use this to describe the N\'{e}ron model 
 of the associated family of intermediate Jacobians (of moduli spaces).

 \begin{note}\label{tor33}
  Denote by $\pi_1:\mc{X} \to \Delta$ a family of projective curves of genus $g \ge 2$
  over the unit disc $\Delta$, smooth over 
  the punctured disc $\Delta^*$ and central fiber isomorphic to an irreducible nodal curve $X_0$ with exactly one node, say at $x_0$. 
  Assume further that $\mc{X}$ is regular. Let $\pi: \widetilde{X}_0 \to X_0$ be the normalization map.
  Fix an invertible sheaf $\mc{L}$ on $\mc{X}$ of relative odd degree, say $d$. 
  Set $\mc{L}_0:=\mc{L}|_{X_0}$, the restriction of $\mc{L}$ to the central fiber.
  Denote by $\widetilde{\mc{L}}_0:=\pi^*\mc{L}_0$.
  
  Denote by $\widetilde{\mc{X}}:=\mr{Bl}_{x_0}\mc{X}$ and by 
  \begin{equation}\label{eq:tor01}
   \widetilde{\pi}_1:\widetilde{\mc{X}} \to \mc{X} \xrightarrow{\pi_1} \Delta.
  \end{equation}
 Note that the central fiber of $\widetilde{\pi}_1$ is the union of two irreducible components, the normalization $\widetilde{X}_0$ of $X_0$ and 
  the exceptional divisor $F \cong \mb{P}^1_{x_0}$ intersecting $\widetilde{X}_0$ at the two points 
  over $x_0$.
  \end{note}

  \subsection{Relative Gieseker moduli space}\label{subsec1}
  Recall, for any smooth, projective curve $Y$ of genus $g$ at least $2$ and an 
  invertible sheaf $\mc{L}'$ on $Y$ of odd degree, there exists a non-singular (fine) moduli 
  space, denoted $M_Y(2,\mc{L}')$, parameterizing rank $2$ semi-stable sheaves 
  on $Y$ with determinant $\mc{L}'$ (see \cite{ind}, \cite{ink3} for basic definitions and results on moduli spaces of sheaves with fixed determinant).
   There exists a relative Gieseker moduli space, denoted $\mc{G}(2, \mc{L})$, parameterizing
   families of rank $2$, semi-stable sheaves defined over families of curves, semi-stably 
   equivalent to $\mc{X}$, with determinant $\mc{L}$. See \cite[\S $3$]{sun1} or \cite[\S 6]{tha} 
   for the precise definition. We omit the precise definition in this article as it is very technical.
   Instead, we recall the necessary properties of the moduli space. 
   
   Note that $\mc{G}(2,\mc{L})$ is regular and there exists a flat, projective morphism 
   \[\pi_2: \mc{G}(2, \mc{L}) \to \Delta\]
   such that:
   \begin{enumerate}
    \item for all $s \in \Delta^*$, $\mc{G}(2,\mc{L})_s:=\pi_2^{-1}(s)=M_{\mc{X}_s}(2,\mc{L}_s)$, where $\mc{L}_s:=\mc{L}|_{\mc{X}_s}$,
    \item the central fiber, denoted $\mc{G}_{X_0}(2,\mc{L}_0):=\pi_2^{-1}(0)$, 
    is a reduced simple normal crossings divisor of $\mc{G}(2,\mc{L})$, consisting of two 
    smooth, irreducible components, say $\mc{G}_0$ and $\mc{G}_1$ with $\mc{G}_1$ (resp. $\mc{G}_0 \cap \mc{G}_1$) is isomorphic to a $\p3$ (resp. $\mb{P}^1 \times \mb{P}^1$)-bundle over $M_{\widetilde{X}_0}(2,\widetilde{\mc{L}}_0)$.
    Moreover, there exists an $\ov{\mr{SL}}_2$-bundle $P_0$ over $M_{\widetilde{X}_0}(2,\widetilde{\mc{L}}_0)$ and closed subvarieties $Z \subset P_0$ and $Z' \subset \mc{G}_0$
    such that 
    \[Z' \cap (\mc{G}_0 \cap \mc{G}_1)= \emptyset \mbox{ and } \mc{G}_0 \backslash Z' \cong P_0 \backslash Z,\]
    where $\ov{\mr{SL}}_2$ is the \emph{wonderful compactification} of $\mr{SL}_2$ defined as 
    \[\ov{\mr{SL}}_2:= \{[M,\lambda] \in \mb{P}(\mr{End}(\mb{C}^2) \oplus \mb{C})| \det(M)=\lambda^2\}.\]
   \end{enumerate}
 See \cite[\S $6$]{tha} for a proof of the above statement (see also \cite[\S $5, 6$]{abe}) and \cite[Definition $3.3.1$]{pezz} for  the general definition of wonderful compactification. Also note that by \cite{K4} the moduli space $\mc{G}(2,\mc{L})_s$ is non-empty for any $s\in \Delta$.

  \subsection{Relative Mumford-Newstead isomorphism}
 Let us consider the relative version of the construction in \cite{mumn}.
  Denote by  \[\mc{W}:=\mc{X}_{\Delta^*} \times_{\Delta^*} \mc{G}(2,\mc{L})_{\Delta^*}\, \mbox{ and }\, \pi_3: \mc{W} \to \Delta^*\] the natural morphism. Recall, 
  $\mc{W}_t:=\pi_3^{-1}(t)=\mc{X}_t \times \mc{G}(2,\mc{L})_t \cong \mc{X}_t \times M_{\mc{X}_t}(2,\mc{L}_t)$, for all $t \in \Delta^*$.
   Using \cite[Theorem $9.1.1$]{pan}, one can check that there exists a (relative) universal bundle $\mc{U}$ over $\mc{W}$
  associated to the (relative) moduli space $\mc{G}(2,\mc{L})_{\Delta^*}$. In particular, 
  for each $t \in \Delta^*$,
  $\mc{U}|_{\mc{W}_t}$ is the universal bundle over $\mc{X}_t \times M_{\mc{X}_t}(2,\mc{L}_t)$ associated to the 
  fine moduli space $M_{\mc{X}_t}(2,\mc{L}_t)$
  (see \cite[Corollary $4.6.6$]{huy}).
  Denote by $\mb{H}^4_{\mc{W}}:=R^4 \pi_{3_*} \mb{Z}_{\mc{W}}$ the local system associated to $\mc{W}$.
  Using the K\"{u}nneth decomposition, we have (see \S \ref{nsec1} for notations)
  \begin{equation}\label{ntor06}
   \mb{H}^4_{\mc{W}}= \bigoplus\limits_i \left(\mb{H}^i_{{\mc{X}}_{\Delta^*}} \otimes \mb{H}^{4-i}_{\mc{G}(2,\mc{L})_{\Delta^*}}\right).
  \end{equation}
  By the Poincar\'{e} duality applied to the local system $\mb{H}^1_{{\mc{X}}_{\Delta^*}}$ (see \cite[\S I.$2.6$]{kuli}), we have
\begin{equation}\label{ner08}
 \mb{H}^1_{{\mc{X}}_{\Delta^*}} \otimes \mb{H}^{3}_{\mc{G}(2,\mc{L})_{\Delta^*}} \stackrel{\mr{PD}}{\cong}\left(\mb{H}^1_{{\mc{X}}_{\Delta^*}}\right)^\vee \otimes \mb{H}^{3}_{\mc{G}(2,\mc{L})_{\Delta^*}} \cong \Hc\left(\mb{H}^1_{{\mc{X}}_{\Delta^*}}, \mb{H}^{3}_{\mc{G}(2,\mc{L})_{\Delta^*}}\right).
\end{equation}
  Denote by $c_2(\mc{U})^{1,3} \in \Gamma\left(\mb{H}^1_{{\mc{X}}_{\Delta^*}} \otimes \mb{H}^{3}_{\mc{G}(2,\mc{L})_{\Delta^*}}\right)$ 
  the image of the second Chern class $c_2(\mc{U}) \in \Gamma(\mb{H}^4_{\mc{W}})$ under the natural projection 
  $\mb{H}^4_{\mc{W}} \to \mb{H}^1_{{\mc{X}}_{\Delta^*}} \otimes \mb{H}^{3}_{\mc{G}(2,\mc{L})_{\Delta^*}}$.
    Then, $c_2(\mc{U})^{1,3}$ induces a homomorphism 
    \[\Phi_{\Delta^*}: \mb{H}^1_{\mc{X}_{\Delta^*}} \to \mb{H}^{3}_{\mc{G}(2,\mc{L})_{\Delta^*}}.\]
  Denote by 
   \begin{equation}\label{ntor07}
    \widetilde{\Phi}_s: H^1(\widetilde{\mc{X}}_s, \mb{Z}) \xrightarrow{\sim} H^3(\mc{G}(2,\mc{L})_s, \mb{Z})
   \end{equation}
   the restriction of $\Phi_{\Delta^*}$ to the point $s \in \Delta^*$. Since $c_2(\mc{U})^{1,3}$ is a (single-valued) global section of 
  $\mb{H}^1_{\widetilde{\mc{X}}_{\Delta^*}} \otimes \mb{H}^{3}_{\mc{G}(2,\mc{L})_{\Delta^*}}$
  (see \cite[Proposition $10.1$]{fult}),
   we have 
   {\small \[\widetilde{\Phi}_s \in \mr{Hom}(H^1(\widetilde{\mc{X}}_s, \mb{Z}), H^3(\mc{G}(2,\mc{L})_s, \mb{Z})) \cong H^1(\widetilde{\mc{X}}_s, \mb{Z})^\vee \otimes H^3(\mc{G}(2,\mc{L})_s, \mb{Z}) \stackrel{\mr{P.D.}}{\cong}
           H^1(\widetilde{\mc{X}}_s, \mb{Z}) \otimes H^3(\mc{G}(2,\mc{L})_s, \mb{Z}) 
           \]}
is monodromy invariant i.e., for all $s \in \Delta^*$, the following diagram is commutative:
   \begin{equation}\label{ner03}
    \begin{diagram}
     H^1(\widetilde{\mc{X}}_s, \mb{Z}) &\rTo^{\widetilde{\Phi}_s}_{\sim}& H^3(\mc{G}(2,\mc{L})_s, \mb{Z}) \\
  \dTo^{T_{\widetilde{\mc{X}}_s}}_{\cong}&\circlearrowleft&\dTo_{T_{\mc{G}(2,\mc{L})_s}}^{\cong}\\
    H^1(\widetilde{\mc{X}}_s, \mb{Z}) &\rTo^{\widetilde{\Phi}_s}_{\sim}&H^3(\mc{G}(2,\mc{L})_s, \mb{Z})  
       \end{diagram}
   \end{equation}
   where $T_{\widetilde{\mc{X}}_s}$ and $T_{\mc{G}(2,\mc{L})_s}$ are the monodromy transformations on $H^1(\widetilde{\mc{X}}_s, \mb{Z})$
 and $H^3(\mc{G}(2,\mc{L})_s, \mb{Z})$, respectively.
  By \cite[Lemma $1$ and Proposition $1$]{mumn}, we conclude that the homomorphism $\Phi_{\Delta^*}$
  is an isomorphism such that the induced 
  isomorphism on the associated vector bundles:
  \[\Phi_{\Delta^*}:\mc{H}^1_{\mc{X}_{\Delta^*}} \xrightarrow{\sim} \mc{H}^{3}_{\mc{G}(2,\mc{L})_{\Delta^*}} \mbox{ satisfies } \Phi_{\Delta^*}(F^p\mc{H}^1_{\mc{X}_{\Delta^*}})= F^{p+1}\mc{H}^{3}_{\mc{G}(2,\mc{L})_{\Delta^*}} \mbox{ for all }p \ge 0.\]
 Therefore, the morphism $\Phi_{\Delta^*}$ induces an isomorphism:
 \begin{equation}\label{ntor10}
  \Phi': \mbf{J}^1_{{\mc{X}}_{\Delta^*}} \xrightarrow{\sim} \mbf{J}^2_{\mc{G}(2,\mc{L})_{\Delta^*}}
 \end{equation}

 \subsection{Limit Mumford-Newstead isomorphism}  
 The isomorphism $\Phi_{\Delta^*}$ can be extended to the entire disc $\Delta$ such that 
 the induced morphism on the central fibers is an isomorphism of limit mixed Hodge structures. 
 Let $\ov{\mc{H}}^1_{\mc{X}_{\Delta^*}}$ and $\ov{\mc{H}}^{3}_{\mc{G}(2,\mc{L})_{\Delta^*}}$ be the canonical extensions of $\mc{H}^1_{\mc{X}_{\Delta^*}}$
  and $\mc{H}^{3}_{\mc{G}(2,\mc{L})_{\Delta^*}}$, respectively. By the uniqueness of the canonical extension, 
  the morphism $\Phi_{\Delta^*}$ extends to the entire disc:
  \[\widetilde{\Phi}:\ov{\mc{H}}_{\widetilde{\mc{X}}}^1 \xrightarrow{\sim} \ov{\mc{H}}^3_{\mc{G}(2,\mc{L})}.\]
  Using the identification \eqref{tor23} and restricting $\widetilde{\Phi}$ to the central fiber, we have an isomorphism:
  \begin{equation}\label{ner11}
   \widetilde{\Phi}_0 : H^1(\widetilde{\mc{X}}_\infty,\mb{Q}) \xrightarrow{\sim} H^3(\mc{G}(2,\mc{L})_\infty,\mb{Q}).
  \end{equation}
   Recall that $\widetilde{\Phi}_0$ is an isomorphism of mixed Hodge structures: 
  \begin{thm}\label{ner01}
   For the extended morphism $\widetilde{\Phi}$, we have 
   $\widetilde{\Phi}(F^p\ov{\mc{H}}_{\widetilde{\mc{X}}}^1)= F^{p+1}\ov{\mc{H}}^3_{\mc{G}(2,\mc{L})}$  for  $p=0,1$  and  
   $\widetilde{\Phi}(\ov{\mb{H}}_{\widetilde{\mc{X}}}^1)=\ov{\mb{H}}^3_{\mc{G}(2,\mc{L})}$.
   Moreover, $\widetilde{\Phi}_0(W_iH^1(\widetilde{\mc{X}}_\infty,\mb{Q}))=W_{i+2}H^3(\mc{G}(2,\mc{L})_\infty,\mb{Q})$
   for all $i \ge 0$.
  \end{thm}
  
 \begin{proof}
   See \cite[Proposition $4.1$]{indpre} for a proof of the statement.
 \end{proof}

 \subsection{Vanishing of the finite group $G_i$}
 Denote by \[\mb{H}^i_{\widetilde{\mc{X}}_{\Delta^*}}:=R^i (\widetilde{\pi}'_1)_* \mb{Z}\, \mbox{ and }\, \mb{H}^i_{\mc{G}(2,\mc{L})_{\Delta^*}}:=R^i (\pi'_2)_* \mb{Z},\]
 the local systems associated to the families $\widetilde{\pi}_1':\widetilde{\mc{X}}_{\Delta^*} \to \Delta^*$ and $\pi'_2:\mc{G}(2,\mc{L})_{\Delta^*} \to \Delta^*$
 which are the restrictions of $\widetilde{\pi}_1$ and $\pi_2$ to $\Delta^*$, respectively.
 For any $s \in \Delta^*$, denote by 
  \[T^i_{\mc{G}(2,\mc{L})_s}: H^i(\mc{G}(2,\mc{L})_s,\mb{Z}) \to  H^i(\mc{G}(2,\mc{L})_s,\mb{Z})\, \mbox{ and }\, 
   T^{\mb{Q},i}_{\mc{G}(2,\mc{L})_s}: H^i(\mc{G}(2,\mc{L})_s,\mb{Q}) \to  H^i(\mc{G}(2,\mc{L})_s,\mb{Q})\]
   the local monodromy transformation associated to the local system  $\mb{H}^i_{\mc{G}(2,\mc{L})_{\Delta^*}}$.
 Recall, by \cite[p. $10$]{atyang} that $H^i(\mc{G}(2,\mc{L})_s,\mb{Z})$ is torsion-free, for all $i \ge 0$.
 As in \eqref{ntor05}, denote by 
   \[G_i:=\frac{(T^{\mb{Q},i}_{\mc{G}(2,\mc{L})_s}-\mr{Id})H^i(\mc{G}(2,\mc{L})_s,\mb{Q}) \cap H^i(\mc{G}(2,\mc{L})_s,\mb{Z})}{(T^{i}_{\mc{G}(2,\mc{L})_s}-\mr{Id})H^i(\mc{G}(2,\mc{L})_s,\mb{Z})}.\]
  We show below that $G_i$ vanishes for all $i \ge 0$ (Theorem \ref{ntor14}). The idea of the proof 
  is to use the isomorphism $\widetilde{\Phi}_0$ as in \eqref{ner11}
  and combine it with Newstead's classical result (see \cite[Theorem 1]{new1})
  on generators of the cohomology ring of $\mc{G}(2,\mc{L})_s$ for $s \in \Delta^*$, to reduce 
  the problem to the Picard-Lefschetz formula associated to the family of curves $\widetilde{\pi}'_1$.
  
  \begin{thm}\label{ntor14}
   The group $G_i=0$ for all $i \ge 0$.
  \end{thm}

  \begin{proof}
   Since $\mc{G}(2,\mc{L})_s$ is rationally connected for all $s \in \Delta^*$, we have $H^1(\mc{G}(2,\mc{L})_s,\mb{Q})=0$. This implies
   \[\mb{H}^1_{\mc{G}(2,\mc{L})_{\Delta^*}}=0 \mbox{ and } G_0=0=G_1.\]
   Let $\mc{W}:=\mc{X}_{\Delta^*} \times_{\Delta^*} \mc{G}(2,\mc{L})_{\Delta^*}$ and $\pi_3: \mc{W} \to \Delta^*$ the natural morphism.
   Let $\mb{H}^i_{\mc{W}}:=R^i \pi_{3_*} \mb{Z}_{\mc{W}}$ be the associated local system.
   By the K\"{u}nneth decomposition, we then have 
   \[\mb{H}^2_{\mc{W}} \cong \mb{H}^2_{\mc{G}(2,\mc{L})_{\Delta^*}} \oplus \mb{H}^2_{\widetilde{\mc{X}}_{\Delta^*}} \mbox{ and } 
   \mb{H}^4_{\mc{W}} \cong \mb{H}^4_{\mc{G}(2,\mc{L})_{\Delta^*}} \oplus \mb{H}^3_{\mc{G}(2,\mc{L})_{\Delta^*}} \otimes \mb{H}^1_{\widetilde{\mc{X}}_{\Delta^*}} \oplus 
   \mb{H}^2_{\mc{G}(2,\mc{L})_{\Delta^*}} \otimes \mb{H}^2_{\widetilde{\mc{X}}_{\Delta^*}}.
       \]
Now, the space of global sections of $\mb{H}^2_{\widetilde{\mc{X}}_{\Delta^*}}$ is generated by the (relative) dual fundamental class $f_{\Delta^*}$ of $\widetilde{\mc{X}}_{\Delta^*}$ i.e.,
for all $s \in \Delta^*$, the restriction of $f_{\Delta^*}$ to the fiber $\widetilde{\mc{X}}_s$ is the dual fundamental class $f_s$ of $\widetilde{\mc{X}}_s$.
As mentioned in the previous section, there exists a (relative) universal bundle $\mc{U}$ over $\mc{W}$
  associated to the (relative) moduli space $\mc{G}(2,\mc{L})_{\Delta^*}$. 
  Now, $c_1(\mc{U})$ and $c_2(\mc{U})$ define global sections of $\mb{H}^2_{\mc{W}}$ and $\mb{H}^4_{\mc{W}}$, respectively (see \cite[Proposition $10.1$]{fult}).
By \cite[Theorem $1$, Corollary $2$]{newt} (see also \cite[p.$338$]{new1}), we have 
\[c_1(\mc{U})=\phi_{\Delta^*}+f_{\Delta^*} \mbox{ and } c_2(\mc{U})=\tau_{\Delta^*}+c_2(\mc{U})^{1,3}+\omega_{\Delta^*} \otimes f_{\Delta^*}\]
 for some  $\phi_{\Delta^*}, \omega_{\Delta^*} \in \Gamma(\mb{H}^2_{\mc{G}(2,\mc{L})_{\Delta^*}})$,
 $c_2(\mc{U})^{1,3} \in \Gamma(\mb{H}^1_{\widetilde{\mc{X}}_{\Delta^*}} \otimes \mb{H}^{3}_{\mc{G}(2,\mc{L})_{\Delta^*}})$ and $\tau_{\Delta^*} \in  \Gamma(\mb{H}^4_{\mc{G}(2,\mc{L})_{\Delta^*}})$.
 Denote by \[\alpha_{\Delta^*}:=2\omega_{\Delta^*}-\phi_{\Delta^*} \mbox{ and } \beta_{\Delta^*}:=\phi^2_{\Delta^*}-4\tau_{\Delta^*}.\]
 Since $\alpha_{\Delta^*}$ and $\beta_{\Delta^*}$ are global sections of $\mb{H}^2_{\mc{G}(2,\mc{L})_{\Delta^*}}$ and $ \mb{H}^4_{\mc{G}(2,\mc{L})_{\Delta^*}}$, respectively, we have 
 \[T^2_{\mc{G}(2,\mc{L})_s}(\alpha_s)=\alpha_s \mbox{ and }T^4_{\mc{G}(2,\mc{L})_s}(\beta_s)=\beta_s,\]
 where $\alpha_s \in H^2(\mc{G}(2, \mc{L})_s, \mb{Z})$ and $\beta_s \in H^4(\mc{G}(2,\mc{L})_s, \mb{Z})$ are the restrictions
 of $\alpha_{\Delta^*}$ and $\beta_{\Delta^*}$, respectively, to the fiber $\widetilde{\mc{X}}_s$.
 Denote by $\psi_i \in H^3(\mc{G}(2,\mc{L})_s, \mb{Z})$ the image of $e_i$ (as in Theorem \ref{ntor11}) under the morphism
 \[H^1(\widetilde{\mc{X}}_\infty, \mb{Z}) \xrightarrow{\sim} H^1(\widetilde{\mc{X}}_s, \mb{Z}) \xrightarrow{\widetilde{\Phi}_s} H^3(\mc{G}(2,\mc{L})_s,\mb{Z}) \mbox{ for } i=1, ..., 2g,\]
 where $\widetilde{\Phi}_s$ is the morphism in \eqref{ner03}.
 Note that $\psi_1, ...., \psi_{2g}$ generate $H^3(\mc{G}(2,\mc{L})_s, \mb{Z})$.
 By \cite[Theorem $1$]{new1}, the cohomology ring $H^*(\mc{G}(2,\mc{L})_s, \mb{Q})$ is generated by $\alpha_s, \beta_s, \psi_1, \psi_2, ...., \psi_{2g}$.
 
 As $H^2(\mc{G}(2,\mc{L})_s, \mb{Z})$ is generated by $\alpha_s$, which is monodromy invariant, we have $G_2=0$.  Since monodromy operator commutes with cup-product (as pullback under continuous morphisms commute
 with cup-product), we have $T^2_{\mc{G}(2,\mc{L})_s}(\alpha_s \cup \alpha_s)=\alpha_s \cup \alpha_s$. Since $H^4(\mc{G}(2,\mc{L})_s, \mb{Q})$ 
 is generated by $\alpha_s \cup \alpha_s$ and $\beta_s$, we can 
 similarly conclude that $G_4=0$.
 
 Denote by $\delta \in H_1(\mc{X}_s,\mb{Z})$ the vanishing cycle associated to the degeneration of curves defined by $\pi_1$ (see \cite[\S $3.2.1$]{v5}).
  Note that $\delta$ is the generator of the kernel of the natural morphism 
 \[H_1(\mc{X}_s,\mb{Z}) \xrightarrow{i_{s_*}} H_1(\mc{X},\mb{Z}) \xrightarrow[\sim]{r_0} H_1({X}_0,\mb{Z}),\]
 where $i_s:\mc{X}_s \to \mc{X}$ is the natural inclusion of fiber and $r_0:\mc{X} \to {X}_0$ is the retraction to the 
 central fiber (see \cite[Corollary $2.17$]{v5}).
 Since  ${X}_0$ is an irreducible nodal curve, the homology group $H_1({X}_0,\mb{Z})$ is torsion-free.
 Therefore, $\delta$ is non-divisible i.e., there does not exist $\delta' \in H_1(\mc{X}_s,\mb{Z})$ such that 
 $n\delta'=\delta$ for some integer $n \not= \pm 1$. 
 Denote by $(-,-)$ the intersection form on $H_1(\mc{X}_s,\mb{Z})$, defined using cup-product (see \cite[\S $7.1.2$]{v4}).
 Since the intersection form $(-,-)$ induces a perfect pairing on $H_1(\mc{X}_s,\mb{Z})$, 
 the non-divisibility of $\delta$ implies that there exists $\gamma \in H_1(\mc{X}_s,\mb{Z})$ such that $(\gamma, \delta)=1$. 
  Recall the Picard-Lefschetz formula, \[T_{\mc{X}_s}(\eta)=\eta+(\delta, \eta) \delta \, \mbox{ for any } \eta \in  H_1(\mc{X}_s,\mb{Z}).\]
 This implies, 
 $(T_{\mc{X}_s}-\mr{Id})H^1(\mc{X}_s,\mb{Q}) \cap H^1(\mc{X}_s,\mb{Z})=\mb{Z} \delta^c=(T_{\mc{X}_s}-\mr{Id})H^1(\mc{X}_s,\mb{Z})$, where $\delta^c$ is the Poincar\'{e} 
 dual to the vanishing cycle $\delta$. Note that $T_{\mc{X}_s}=T_{\widetilde{\mc{X}}_s}$ (as $\mc{X}_s=\widetilde{\mc{X}}_s$ for all $s \in \Delta^*$).
 Since $\widetilde{\Phi}_s$ is an isomorphism, the diagram \eqref{ner03} implies that 
 \[(T_{\mc{G}(2,\mc{L})_s}^3-\mr{Id})H^3(\mc{G}(2,\mc{L})_s, \mb{Z})=\widetilde{\Phi}_s \circ (T_{\widetilde{\mc{X}}_s}-\mr{Id})H^1(\widetilde{\mc{X}}_s, \mb{Z})=\mb{Z}\widetilde{\Phi}_s(\delta^c).\]
 Similarly, we have 
 \[(T_{\mc{G}(2,\mc{L})_s}^{\mb{Q},3}-\mr{Id})H^3(\mc{G}(2,\mc{L})_s, \mb{Q}) \cap H^3(\mc{G}(2,\mc{L})_s, \mb{Z})=\widetilde{\Phi}_s \circ (T_{\widetilde{\mc{X}}_s}-\mr{Id})H^1(\widetilde{\mc{X}}_s, \mb{Q}) \cap 
 \widetilde{\Phi}_s(H^1(\widetilde{\mc{X}}_s, \mb{Z}))=\]
 \[= \widetilde{\Phi}_s \circ ((T_{\widetilde{\mc{X}}_s}-\mr{Id})H^1(\widetilde{\mc{X}}_s, \mb{Q}) \cap H^1(\widetilde{\mc{X}}_s, \mb{Z}))=\mb{Z}\widetilde{\Phi}_s(\delta^c).\]
 This implies $G_3=0$.
 
\[(T^{\mb{Q},i}_{\mc{G}(2,\mc{L})_s}-\mr{Id})H^i(\mc{G}(2,\mc{L})_s, \mb{Q}) \cap V_i=\widetilde{\Phi}_s(\delta^c) \cup H^{i-3}(\mc{G}(2,\mc{L})_s, \mb{Z})=(T^i_{\mc{G}(2,\mc{L})_s}-\mr{Id})V_i, \mbox{ where }\]
$V_i:=H^i(\mc{G}(2,\mc{L})_s, \mb{Z})$. This implies $G_i=0$ for all $i \ge 5$. This proves the theorem.
  \end{proof}

  Expanding $N_{2i-1}=\log(T^{2i-1}_{\mc{G}(2,\mc{L})_s})$, one observes that $N_{2i-1}=N'_{2i-1} \circ (T^{2i-1}_{\mc{G}(2,\mc{L})_s}-\mr{Id})$,
where $N'_{2i-1}$ is of the form $\mr{Id}+A$ for a nilpotent operator $A$. This implies that $N'_{2i-1}$ induces 
an automorphism of $H^{2i-1}(\mc{G}(2,\mc{L})_s, \mb{Q})$. It is then easy to check that $(T^{2i-1}_{\mc{G}(2,\mc{L})_s}-\mr{Id})^2=0$
if and only if $N_{2i-1}^2=0$. Under the natural identification $H^{2i-1}(\mc{G}(2,\mc{L})_s,\mb{Q}) \cong H^{2i-1}(\mc{G}(2,\mc{L})_\infty,\mb{Q})$,
Proposition \ref{ntor03} along with \eqref{t02} implies that
\[\ker(N_{2i-1}) \cong \Ima(\mr{sp}_{2i-1}) \cong W_{2i-1}H^{2i-1}(\mc{G}(2,\mc{L})_\infty, \mb{Q})\]
which contains $W_{2i-2} H^{2i-1}(\mc{G}(2,\mc{L})_\infty, \mb{Q})=N_{2i-1}(H^{2i-1}(\mc{G}(2,\mc{L})_\infty, \mb{Q}))$ (Theorem \ref{tor25}). 
Hence $N_{2i-1}^2=0$, thereby $(T^{2i-1}_{\mc{G}(2,\mc{L})_s}-\mr{Id})^2=0$.
   By Proposition \ref{ntor03}, the limit mixed Hodge structure on $H^{2i-1}(\mc{G}(2,\mc{L})_\infty, \mb{Q})$ has weight filtration 
   $W_{2i-2} \subset W_{2i-1} \subset W_{2i}$.
  It was shown by Clemens in \cite[Corollary $3.24$]{cle} and Saito in \cite[Theorem $2.8$]{sai} that in this case
  there exists a Hausdorff topological space \[\rho':\mbf{J}^i_e \to \Delta\] extending $\widetilde{\rho}:\mbf{J}^i_{\mc{G}(2,\mc{L})_{\Delta^*}} \to \Delta^*$,
  which they called the \emph{N\'{e}ron model of $\widetilde{\rho}$}. The central fiber $(\rho')^{-1}(0)$ is isomorphic as a complex Lie group
  to $J'_i$ as in \eqref{ntor13}, after replacing $\mc{Y}$ by $\mc{G}(2,\mc{L})$  (see \cite[Proposition II.A.$8$]{green}).
  We can then prove:
  \begin{cor}\label{tor02}
   The central fiber of $\rho'$ above is isomorphic as a complex Lie group to the central fiber $\left(\ov{\mbf{J}}^i_{\mc{G}(2, \mc{L})}\right)_0$ of the GGK-N\'{e}ron model 
   $\ov{\mbf{J}}^i_{\mc{G}(2, \mc{L})}$.
  \end{cor}

  \begin{proof}
     This follows immediately from Theorems \ref{ntor08} and \ref{ntor14}.
  \end{proof}

 \section{N\'{e}ron model of intermediate Jacobian associated to moduli spaces}\label{nsec4}
 
  Notations as in Notation \ref{tor33} and \S \ref{subsec1}. In this section, we study the N\'{e}ron model of 
  families of intermediate Jacobians associated to the family of moduli spaces given by $\pi_2$.
  
  \subsection{Comparing $\mbf{\mr{Gr}^W_{2i-1}H^{2i-1}(\mc{G}(2,\mc{L})_\infty, \mb{Q})}$ and $\mbf{H^{2i-1}(M_{\widetilde{X}_0}(2,\widetilde{\mc{L}}_0), \mb{Q})}$}
  We first consider the case $i=2$.
  Using \cite[Proposition $1$]{mumn}, there exists an isomorphism of pure Hodge structures:
 \[\Phi_0':H^1(\widetilde{X}_0, \mb{Z}) \xrightarrow{\sim} H^3(M_{\widetilde{X}_0}(2,\widetilde{\mc{L}}_0), \mb{Z}).\]
  The Mayer-Vietoris sequence associated to the central fiber $\widetilde{\mc{X}}_0$ (notation as in Notation \ref{tor33}) is
{\small \[0 \to H^0(\widetilde{\mc{X}}_0, \mb{Z}) \to H^0(F, \mb{Z}) \oplus H^0(\widetilde{X}_0, \mb{Z}) \to H^0(F \cap \widetilde{X}_0, \mb{Z}) \to H^1(\widetilde{\mc{X}}_0, \mb{Z}) \to H^1(F, \mb{Z}) \oplus H^1(\widetilde{X}_0, \mb{Z}) \to 0.\]}
Since $H^1(F,\mb{Z})=0$, this gives us the short exact sequence:
 \begin{equation}\label{ner10}
 0 \to \mb{Z} \xrightarrow{p} H^1(\widetilde{\mc{X}}_0, \mb{Z}) \xrightarrow{q} H^1(\widetilde{X}_0,\mb{Z}) \to 0,
\end{equation}
 inducing isomorphisms $\mb{Q} \stackrel{p}{\cong} \mr{Gr}^W_0H^1(\widetilde{\mc{X}}_0, \mb{Q})$ and $\mr{Gr}^W_1H^1(\widetilde{\mc{X}}_0, \mb{Q}) \stackrel{q}{\cong} H^1(\widetilde{X}_0,\mb{Q})$.
 Using the short exact sequence \eqref{ner10} and Theorem \ref{ner01}, we have the composed morphism 
  \[\Phi_1: \mr{Gr}^W_3 H^3(\mc{G}(2,\mc{L})_\infty, \mb{Q}) \to H^3(M_{\widetilde{X}_0}(2,\widetilde{\mc{L}}_0), \mb{Q}) \mbox{ defined by }\]
    {\small \[\mr{Gr}^W_3 H^3(\mc{G}(2,\mc{L})_\infty, \mb{Q}) \xleftarrow[\sim]{\widetilde{\Phi}_0} \mr{Gr}^W_1H^1(\widetilde{\mc{X}}_\infty, \mb{Q})
   \xleftarrow[\sim]{\mr{sp}_1} \mr{Gr}^W_1H^1(\widetilde{\mc{X}}_0,\mb{Q}) \xrightarrow[\sim]{q} H^1(\widetilde{X}_0,\mb{Q}) \xrightarrow[\sim]{\Phi_0'}
   H^3(M_{\widetilde{X}_0}(2,\widetilde{\mc{L}}_0), \mb{Q}),
  \]}
where the first isomorphism is given by \eqref{ner11} and the second isomorphism follows directly from Proposition \ref{ntor03}.
  By Theorem \ref{ner01}, $\widetilde{\Phi}_0$ is an isomorphism of pure Hodge structures. Also, note that the last three morphisms in the 
  composed morphism $\Phi_1$ are morphisms of pure Hodge structures. Therefore, $\Phi_1$ is an isomorphism of pure Hodge structures.
  In general,
  
  \begin{prop}\label{ner12}
     There exists a morphism (induced by $\Phi_1$) {\small \[\Phi^{(i)}_1:\frac{\mr{Gr}^W_{2i-1}H^{2i-1}(\mc{G}(2, \mc{L})_\infty, \mb{C})}{F^i \mr{Gr}^W_{2i-1}H^{2i-1}(\mc{G}(2, \mc{L})_\infty, \mb{C})+\mr{Gr}^W_{2i-1}H^{2i-1}(\mc{G}(2, \mc{L})_\infty, \mb{Z})}
   \to J^{i-3}(M_{\widetilde{X}_0}(2, \widetilde{\mc{L}}_0)) \times J^{i}(M_{\widetilde{X}_0}(2, \widetilde{\mc{L}}_0)),\]}
where $\mr{Gr}^W_{2i-1}H^{2i-1}(\mc{G}(2, \mc{L})_\infty, \mb{Z})=\mr{Gr}^W_{2i-1}H^{2i-1}(\mc{G}(2, \mc{L})_\infty, \mb{Q}) \cap H^{2i-1}(\mc{G}(2, \mc{L})_\infty, \mb{Z}).$ Moreover, the morphism is an isomorphism for $1 \le i \le \max\{2, g-1\}$. 
   \end{prop}

\begin{proof}
  Denote by $\psi^\infty_i:=\widetilde{\Phi}_0(e_i)$, where $e_i \in H^1(\widetilde{\mc{X}}_\infty, \mb{Z})$ as in Theorem \ref{ntor11}, $1 \le i \le 2g$
  and $\widetilde{\Phi}_0$ as in \eqref{ner11}. Fix $s \in \Delta^*$.
Let $\psi_i \in H^3(\mc{G}(2,\mc{L})_s, \mb{Z})$ be the image of $\psi_i^\infty$ under the natural isomorphism 
\begin{equation}\label{ntor15}
 H^j(\mc{G}(2,\mc{L})_\infty, \mb{Z}) \xrightarrow{\sim} H^j(\mc{G}(2,\mc{L})_s, \mb{Z}), \mbox{ for all } j \ge 0.
\end{equation}
Using \cite[Theorem $1$]{new1}, one can observe that for any $s \in \Delta^*$, 
there exist elements \[\alpha_s \in H^{1,1}(\mc{G}(2,\mc{L})_s, \mb{Z}) \mbox{ and } \beta_s \in 
 H^{2,2}(\mc{G}(2,\mc{L})_s, \mb{Z})\] such that the cohomology ring $H^*(\mc{G}(2,\mc{L})_s, \mb{Q})$ 
 is generated by $\alpha_s, \beta_s, \psi_1, \psi_2, ..., \psi_{2g}$.
  Denote by \[\alpha_\infty \in H^2(\mc{G}(2,\mc{L})_\infty, \mb{Z}) \mbox{ and } \beta_\infty \in H^4(\mc{G}(2,\mc{L})_\infty, \mb{Z})\]
the preimage of $\alpha_s$ and  $\beta_s$, respectively, under the natural isomorphism \eqref{ntor15}.
 It is immediate that the cohomology ring $H^*(\mc{G}(2,\mc{L})_\infty, \mb{Q})$ is generated by 
 $\alpha_\infty, \beta_\infty, \psi_{1}^\infty, \psi_{2}^\infty, ..., \psi_{2g}^\infty$.
 Since $H^2(\mc{G}(2,\mc{L})_\infty, \mb{Z})$ (resp. $H^4(\mc{G}(2,\mc{L})_\infty, \mb{Z})$) is one (resp. two) dimensional, generated by 
 $\alpha_\infty$ (resp. $\alpha_\infty^2$ and $\beta_\infty$), we conclude that (use cup-product is a morphism of MHS)
 \[\alpha_\infty \in H^{1,1}(\mc{G}(2,\mc{L})_\infty, \mb{Z}) \mbox{ and } \beta_\infty \in H^{2,2}(\mc{G}(2,\mc{L})_\infty, \mb{Z}).\]
  Since cup-product morphism is a morphism of mixed Hodge structures, Theorem \ref{ner01} along with \cite[Remark $5.3$]{kingn} imply that 
  a $\mb{Q}$-basis of 
 $\mr{Gr}^W_i H^i(\mc{G}(2,\mc{L})_\infty, \mb{Q})$ consists of monomials of the form 
 $\psi_g^\infty.\psi_{2g}^\infty(\alpha_\infty^{i_1}\beta_\infty^{j_1}\psi^\infty_{k_1}\psi^\infty_{k_2}...\psi^\infty_{k_m})$
 and $\alpha_\infty^{i_1'}\beta_\infty^{j_1'}\psi^\infty_{k'_1}\psi^\infty_{k'_2}...\psi^\infty_{k'_r}$ such that 
 for any $1 \le t \le r$, $k'_t \not\in \{g, 2g\}$, $m+i_1+2<g$, $m+j_1+2<g$, $i_1'+r<g$, $j_1'+r<g$, $2i_1+4j_1+3(m+2)=i$ and $2i_1'+4j_1'+3r=i$.
 By \cite[Theorem $1$]{new1}, there exists $\alpha_0 \in H^2(M_{\widetilde{X}_0}(2,\widetilde{\mc{L}}_0), \mb{Z})$
 and $\beta_0  \in H^4(M_{\widetilde{X}_0}(2,\widetilde{\mc{L}}_0), \mb{Z})$ such that $\alpha_0$ (resp. $\alpha_0^2, \beta_0$)
 generates $H^2(M_{\widetilde{X}_0}(2,\widetilde{\mc{L}}_0), \mb{Q})$ (resp. $H^4(M_{\widetilde{X}_0}(2,\widetilde{\mc{L}}_0), \mb{Q})$).
 We can then define:
 \[\tau_i: \mr{Gr}^W_i H^i(\mc{G}(2,\mc{L})_\infty, \mb{Q}) \rightarrow H^{i-6}(M_{\widetilde{X}_0}(2,\widetilde{\mc{L}}_0), \mb{Q}) \oplus H^{i}(M_{\widetilde{X}_0}(2,\widetilde{\mc{L}}_0), \mb{Q}) \mbox{ by }\]
  \begin{align*}
  \tau_i(\psi_g^\infty.\psi_{2g}^\infty(\alpha_\infty^{i_1}\beta_\infty^{j_1}\psi^\infty_{k_1}\psi^\infty_{k_2}...\psi^\infty_{k_m}))& =(\alpha_0^{i_1}\beta_0^{j_1}\Phi_1(\psi^\infty_{k_1})\Phi_1(\psi^\infty_{k_2})... \Phi_1(\psi^\infty_{k_m})) \oplus 0 \mbox{ and }\\
 \tau_i(\alpha_\infty^{i'_1}\beta_\infty^{j'_1}\psi^\infty_{k'_1}\psi^\infty_{k'_2}...\psi^\infty_{k'_r})& =0 \oplus (\alpha_0^{i'_1}\beta_0^{j'_1}\Phi_1(\psi^\infty_{k'_1})\Phi_1(\psi^\infty_{k'_2})... \Phi_1(\psi^\infty_{k'_r})), \mbox{ if } k'_t \not\in\{g, 2g\}.
 \end{align*}
 As cup-product is a morphism of Hodge structures and $\psi_g^\infty \psi_{2g}^\infty$ is of Hodge type $(3,3)$, one can check that $\tau_i$ 
 is a morphism of pure Hodge structures in the sense that $\tau_i$ maps Hodge type $(p, i-p)$ to $(p-3, i-p-3) \oplus (p, i-p)$ for all $i \ge 0$.
 Note that for $1 \le i \le \max\{3, 2g-3\}$, the above inequalities imply that $m+i_1<g-1, m+j_1<g-1, i_1'+r<g-1$ and $j_1'+r<g-1$.
 Using \cite[Remark $5.3$]{kingn} once again, we conclude that $\tau_i$ is an isomorphism for $1 \le i \le \max\{3, 2g-3\}$. 
  This induces an isomorphism
{\small \[\frac{\mr{Gr}^W_{2i-1}H^{2i-1}(\mc{G}(2, \mc{L})_\infty, \mb{C})}{F^i \mr{Gr}^W_{2i-1}H^{2i-1}(\mc{G}(2, \mc{L})_\infty, \mb{C})+\mr{Gr}^W_{2i-1}H^{2i-1}(\mc{G}(2, \mc{L})_\infty, \mb{Z})}
  \xrightarrow{\tau_{2i-1}} J^{i-3}(M_{\widetilde{X}_0}(2, \widetilde{\mc{L}}_0)) \times J^{i}(M_{\widetilde{X}_0}(2, \widetilde{\mc{L}}_0))
 \]}
for $1 \le i \le \max\{2, g-1\}$. This proves the proposition.
 \end{proof}
 
 \subsection{Central fiber of the N\'{e}ron model}
 Using \eqref{ntor04}, we have a flat family $\widetilde{\rho}: \mbf{J}^i_{\mc{G}(2,\mc{L})_{\Delta^*}} \to \Delta^*$ such that  
 $(\widetilde{\rho})^{-1}(s)=J^i(\mc{G}(2,\mc{L})_s)$ for each  $s \in \Delta^*.$
    By Theorem \ref{ntor08}, there exists a N\'{e}ron model
  \[\ov{\rho}:\ov{\mbf{J}}^i_{\mc{G}(2, \mc{L})} \to \Delta\]
  holomorphically extending the family $\widetilde{\rho}$. We now describe the central fiber 
  of the N\'{e}ron model in terms of the intermediate Jacobian of the Jacobian $\mr{Jac}(\widetilde{X}_0)$ of $\widetilde{X}_0$.
 
 We briefly discuss the idea of the proof of Theorem \ref{ntor12} below.
  The first step is to apply Theorem \ref{ntor14} to Theorem \ref{ntor08}, to express the central fiber 
  of $\ov{\mbf{J}}^i_{\mc{G}(2, \mc{L})}$ as a quotient of \[W_{2i-1}H^{2i-1}(\mc{G}(2,\mc{L})_\infty, \mb{Q}).\]
  This quotient of $W_{2i-1}H^{2i-1}(\mc{G}(2,\mc{L})_\infty, \mb{Q})$ naturally induces a quotient of 
  \[\mr{Gr}^W_{2i-1}H^{2i-1}(\mc{G}(2,\mc{L})_\infty, \mb{Q}),\] which we show is isomorphic to a product of 
  intermediate Jacobians of $\mr{Jac}(\widetilde{X}_0)$ as given in the statement of Theorem \ref{ntor12}. Here we 
  use Proposition \ref{ner12}. As a result, we can view the central fiber of the N\'{e}ron model 
  as a fibration over this product of intermediate Jacobians of $\mr{Jac}(\widetilde{X}_0)$.
  The fiber of the resulting morphism arises as a natural subquotient of a quotient of $W_{2i-1}H^{2i-1}(\mc{G}(2,\mc{L})_\infty, \mb{Q})$
  induced by the natural inclusion $W_{2i-2}H^{2i-1}(\mc{G}(2,\mc{L})_\infty, \mb{Q}) \hookrightarrow
  W_{2i-1}H^{2i-1}(\mc{G}(2,\mc{L})_\infty, \mb{Q})$. In order to give a more explicit description of the fiber, 
  we prove that $W_{2i-2}H^{2i-1}(\mc{G}(2,\mc{L})_\infty, \mb{Q})$ can be 
  identified with $H^{2i-4}(M_{\widetilde{X}_0}(2,\widetilde{\mc{L}}_0), \mb{Q})$, as pure Hodge structures, thereby proving
Theorem \ref{ntor12} below.
  
  \begin{thm}\label{ntor12}
 There exists a  morphism $\eta^{(i)}: \left(\ov{\mbf{J}}^i_{\mc{G}(2, \mc{L})}\right)_0 \to \prod\limits_{k=1}^{[\frac{g}{2}]} J^k(\mr{Jac}(\widetilde{X}_0))^{d_{i,k}}$ such that \[\mbox{ for } 1 \le i \le \max\{2, g-1\},\, \eta^{(i)} \mbox{ is surjective and } pure \left(\left(\ov{\mbf{J}}^i_{\mc{G}(2, \mc{L})}\right)_0\right) \stackrel{\eta^{(i)}}{\cong} \prod\limits_{k=1}^{[\frac{g}{2}]} J^k(\mr{Jac}(\widetilde{X}_0))^{d_{i,k}},\]
 where $d_{i,k}$ is the coefficient of $t^{i-3k+1}$ of the polynomial
 \[(1+t^3)(1+t+t^2+...+t^{g-1-2k})(1+t^2+t^4+...+t^{2(g-1)-4k}).\]
 Moreover, for $1 \le i \le \max\{2,g-1\}$, we have \[\ker \eta^{(i)} \cong (H^{2i-4}(M_{\widetilde{X}_0}(2,\widetilde{\mc{L}}_0), \mb{C}))/(F^{i-1} H^{2i-4}(M_{\widetilde{X}_0}(2,\widetilde{\mc{L}}_0), \mb{C})
+H^{2i-4}(M_{\widetilde{X}_0}(2,\widetilde{\mc{L}}_0), \mb{Z})).\]
  \end{thm}

\begin{proof}
Recall, \cite[Corollary $2.10$]{bano} states that $J^i(M_{\widetilde{X}_0}(2, \widetilde{\mc{L}}_0)) \cong \prod\limits_{k=1}^{[\frac{g}{2}]} J^k(\mr{Jac}(\widetilde{X}_0))^{c_{i,k}},$
 where $c_{i,k}$ is the coefficient of $t^{i-3k+1}$ of the polynomial 
 \[g(t):=(1+t+t^2+...+t^{g-1-2k})(1+t^2+t^4+...+t^{2g-2-4k}).\]
 This implies, 
 \[J^{i-3}(M_{\widetilde{X}_0}(2, \widetilde{\mc{L}}_0)) \times J^i(M_{\widetilde{X}_0}(2, \widetilde{\mc{L}}_0)) \cong \prod\limits_{k=1}^{[\frac{g}{2}]} J^k(\mr{Jac}(\widetilde{X}_0))^{d_{i,k}},\]
 where $d_{i,k}$ is the coefficient of $t^{i-3k+1}$ of the polynomial $(1+t^3)g(t)$.
  Let $T_i$ be the monodromy automorphism  as in \eqref{int01}, after replacing $\mc{Y}$ by $\mc{G}(2,\mc{L})$.
  Let $T_{i,\mb{C}}$ be the induced automorphism on $H^i(\mc{G}(2, \mc{L})_\infty, \mb{C})$ and $N_{i,\mb{C}}:=\log(T_{i,\mb{C}})$.
  By \eqref{t02}, we have 
 \[\ker(T_i-\mr{Id}) \cap H^i(\mc{G}(2,\mc{L})_\infty, \mb{Z})=\mr{sp}_i(H^i(\mc{G}_{X_0}(2,\mc{L}_0), \mb{Z})),\]
 where $\mr{sp}_i$ is the specialization morphism as in Proposition \ref{ntor03}.
  Using \eqref{ntor13} combined with Theorems \ref{ntor08} and \ref{ntor14}, we then have 
 \[\left(\ov{\mbf{J}}^i_{\mc{G}(2, \mc{L})}\right)_0 \cong \frac{\ker N_{2i-1, \mb{C}}}{F^i \ker  N_{2i-1, \mb{C}}+ \left(\ov{\mb{H}}^{2i-1}_{\mc{G}(2,\mc{L})}\right)_0}
  \cong \frac{S_{2i-1}}{F^iS_{2i-1}+\mr{sp}_{2i-1}(H^{2i-1}(\mc{G}_{X_0}(2,\mc{L}_0), \mb{Z}))},\]
where $S_{2i-1}:=\mr{sp}_{2i-1}(H^{2i-1}(\mc{G}_{X_0}(2,\mc{L}_0), \mb{C}))$ and the last isomorphism follows from the invariant cycle theorem.
 By Proposition \ref{ntor03}, $\mr{sp}_{2i-1}(H^{2i-1}(\mc{G}_{X_0}(2,\mc{L}_0), \mb{Q}))=W_{2i-1}H^{2i-1}(\mc{G}(2,\mc{L})_\infty, \mb{Q})$. Hence, 
 $\mr{sp}_{2i-1}(H^{2i-1}(\mc{G}_{X_0}(2,\mc{L}_0), \mb{Z}))$ coincides with 
  \[W_{2i-1}H^{2i-1}(\mc{G}(2,\mc{L})_\infty, \mb{Z}):=W_{2i-1}H^{2i-1}(\mc{G}(2,\mc{L})_\infty, \mb{Q}) \cap H^{2i-1}(\mc{G}(2,\mc{L})_\infty, \mb{Z}).\]
Therefore, 
\[\left(\ov{\mbf{J}}^i_{\mc{G}(2, \mc{L})}\right)_0 \cong \frac{W_{2i-1}H^{2i-1}(\mc{G}(2, \mc{L})_\infty, \mb{C})}{F^i W_{2i-1}H^{2i-1}(\mc{G}(2, \mc{L})_\infty, \mb{C})+W_{2i-1}H^{2i-1}(\mc{G}(2, \mc{L})_\infty, \mb{Z})}.\]
Proposition \ref{ner12} implies that the natural projection morphism from $W_{2i-1}H^{2i-1}(\mc{G}(2, \mc{L})_\infty, \mb{C})$ to $\mr{Gr}^W_{2i-1}H^{2i-1}(\mc{G}(2, \mc{L})_\infty, \mb{C})$
induces a morphism 
$\eta^{(i)}: \left(\ov{\mbf{J}}^i_{\mc{G}(2, \mc{L})}\right)_0 \to \prod\limits_{k=1}^{[\frac{g}{2}]} J^k(\mr{Jac}(\widetilde{X}_0))^{d_{i,k}}$
such that  \[ \mbox{ for } 1 \le i \le \max\{2, g-1\},\,  \eta^{(i)} \mbox{ is surjective and } pure\left(\left(\ov{\mbf{J}}^i_{\mc{G}(2, \mc{L})}\right)_0\right) \stackrel{\eta^{(i)}}{\cong} \prod\limits_{k=1}^{[\frac{g}{2}]} J^k(\mr{Jac}(\widetilde{X}_0))^{d_{i,k}}.\]
For $1 \le i \le \max\{2, g-1\}$, the kernel of the morphism $\eta^{(i)}$ is isomorphic to \[\frac{W_{2i-2} H^{2i-1}(\mc{G}(2,\mc{L})_\infty, \mb{C})}{F^i W_{2i-2} H^{2i-1}(\mc{G}(2,\mc{L})_\infty, \mb{C})+W_{2i-2} H^{2i-1}(\mc{G}(2,\mc{L})_\infty, \mb{Z})}.\]
Let $\alpha_\infty \in H^2(\mc{G}(2,\mc{L})_\infty), \beta_\infty \in H^4(\mc{G}(2,\mc{L})_\infty), \alpha_0 \in H^2(M_{\widetilde{X}_0}(2,\widetilde{\mc{L}}_0)), 
 \beta_0 \in H^4(M_{\widetilde{X}_0}(2,\widetilde{\mc{L}}_0))$ and $\psi_1^\infty, \psi_2^\infty, ..., \psi_{2g}^\infty \in H^3(\mc{G}(2,\mc{L})_\infty)$
 as defined in the proof of Proposition \ref{ner12}. Using Theorem \ref{ner01}, $\psi_g^\infty$ generates $W_2H^3(\mc{G}(2,\mc{L})_\infty, \mb{Q})$.
 As cup-product is a morphism of mixed Hodge structures, it is then easy to check that 
$W_{2i-2}H^{2i-1}(\mc{G}(2,\mc{L})_\infty, \mb{Q})$ is $\mb{Q}$-generated by monomials of the form 
$\alpha_\infty^{i_1}\beta_\infty^{i_2}\psi_g^\infty\psi_{j_1}^\infty\psi_{j_2}^\infty...\psi_{j_k}^\infty$ with $j_t \not= 2g$ for all $1 \le t \le k$ (use \cite[Remark $5.3$]{kingn}).
Define the morphism \[\tau':W_{2i-2}H^{2i-1}(\mc{G}(2,\mc{L})_\infty, \mb{Q}) \to H^{2i-4}(M_{\widetilde{X}_0}(2,\widetilde{\mc{L}}_0), \mb{Q}) \mbox{ as }\]
  \[\tau'(\alpha_\infty^{i_1}\beta_\infty^{i_2}\psi_g^\infty\psi_{j_1}^\infty\psi_{j_2}^\infty...\psi_{j_k}^\infty)=\alpha_0^{i_1}\beta_0^{i_2}\Phi_1(\psi_{j_1}^\infty)\Phi_1(\psi_{j_2}^\infty)...\Phi_1(\psi_{j_k}^\infty)\]
and extend linearly. Since $\psi_g^\infty$ is of Hodge type $(1,1)$ and $\Phi_1$ is an isomorphism of Hodge structures, 
it is easy to check that $\tau'$ is an isomorphism of pure Hodge structures which sends Hodge type $(p,2i-2-p)$
to $(p-1, 2i-3-p)$ (use $W_{2i-2}H^{2i-1}(\mc{G}(2,\mc{L})_\infty, \mb{Q}) = \mr{Gr}^W_{2i-2}H^{2i-1}(\mc{G}(2,\mc{L})_\infty, \mb{Q})$
by Proposition \ref{ntor03}, hence pure). Therefore,
 \[\frac{W_{2i-2} H^{2i-1}(\mc{G}(2,\mc{L})_\infty, \mb{C})}{F^i W_{2i-2} H^{2i-1}(\mc{G}(2,\mc{L}), \mb{C})+W_{2i-2} H^{2i-1}(\mc{G}(2,\mc{L})_\infty, \mb{Z})}
\stackrel{\tau'}{\cong}
\frac{H^{2i-4}}{F^{i-1} H^{2i-4}
+H^{2i-4}(M_{\widetilde{X}_0}(2,\widetilde{\mc{L}}_0), \mb{Z})},\]
where $H^{2i-4}:=H^{2i-4}(M_{\widetilde{X}_0}(2,\widetilde{\mc{L}}_0), \mb{C})$. 
This proves the theorem.
\end{proof}
 
 \begin{rem}\label{rem:ner01}
  Note that $\left(\ov{\mbf{J}}^1_{\mc{G}(2, \mc{L})}\right)_0=0$.
  The theorem immediately tells us that the central fiber of the N\'{e}ron model is \emph{never} an abelian variety. However, we can show that:
 \end{rem}

 \begin{cor}\label{corsemiab}
  For $i=2$, the central fiber $\left(\ov{\mbf{J}}^i_{\mc{G}(2, \mc{L})}\right)_0$ is a semi-abelian variety. Moreover, for $g \ge 5$ and 
  $i=2, 3, 4$, $pure\left(\ov{\mbf{J}}^i_{\mc{G}(2, \mc{L})}\right)_0$ is an abelian variety and 
  $\left(\ov{\mbf{J}}^i_{\mc{G}(2, \mc{L})}\right)_0$ is a semi-abelian variety.
 \end{cor}

 \begin{proof}
   Recall, {\small $H^0(M_{\widetilde{X}_0}(2,\widetilde{\mc{L}}_0),\mb{C}) \cong \mb{C} \cong H^2(M_{\widetilde{X}_0}(2,\widetilde{\mc{L}}_0),\mb{C})$}
and $H^4(M_{\widetilde{X}_0}(2,\widetilde{\mc{L}}_0),\mb{C}) \cong \mb{C}^{\oplus 2}$ is concentrated in the $(2,2)$-Hodge type (see \cite[Theorem $1$]{new1}).
Denote by  
\[K_i:=\frac{H^{2i-4}(M_{\widetilde{X}_0}(2,\widetilde{\mc{L}}_0),\mb{C})}{F^{i-1}H^{2i-4}(M_{\widetilde{X}_0}(2,\widetilde{\mc{L}}_0),\mb{C})+H^{2i-4}(M_{\widetilde{X}_0}(2,\widetilde{\mc{L}}_0),\mb{Z})}\]
It is then easy to check that $K_i \cong \mb{C}^*$ for $i=2, 3$ and $K_4 \cong (\mb{C}^*)^{\oplus 2}$. 
Notations as in Theorem \ref{ntor12}. Note that for $i=2, 3$, we have $d_{i,1}=1$ and $d_{i,j}=0$ for $j \not= 1$. For $i=4$, we have $d_{i,1}=2$ and $d_{i,j}=0$ for $j \not= 1$.
Using \cite[\S $1.1$, \S $1.4$]{birk} observe that 
  \[J^1(\mr{Jac}(\widetilde{X}_0))= \frac{H^1(\mr{Jac}(\widetilde{X}_0), \mb{C})}{F^1 H^1(\mr{Jac}(\widetilde{X}_0), \mb{C}) \oplus H^1(\mr{Jac}(\widetilde{X}_0), \mb{Z})}
   =\frac{H^1(\widetilde{X}_0,\mb{C})}{F^1 H^1(\widetilde{X}_0,\mb{C}) \oplus H^1(\widetilde{X}_0,\mb{Z})}=\mr{Jac}(\widetilde{X}_0).
  \]
  Hence, $J^1(\mr{Jac}(\widetilde{X}_0))$ is an abelian variety. As product of abelian varieties is again an abelian variety,
  Theorem \ref{ntor12} implies that $pure\left(\ov{\mbf{J}}^i_{\mc{G}(2, \mc{L})}\right)_0$ is an abelian variety  and 
  $\left(\ov{\mbf{J}}^i_{\mc{G}(2, \mc{L})}\right)_0$ is an extension of an abelian variety by finitely many copies of $\mb{C}^*$, hence is
  a semi-abelian variety for $i=2$ for any $g \ge 2$ and $i=2, 3, 4$ for $g 
  \ge 5$. This proves the corollary.
 \end{proof}
 
 \section{Intermediate Jacobians of moduli spaces over nodal curves}\label{nsec5}
 
 It is well-known that moduli spaces of semi-stable sheaves with coprime rank and degree over smooth, projective curves are projective 
 and non-singular. So, the intermediate Jacobian of such a moduli space is well-defined and is in fact an abelian variety. In this section, we introduce
 (generalized) intermediate Jacobian of the moduli space $\mc{G}_{X_0}(2,\mc{L}_0)$ defined in \S \ref{subsec1}.
 We describe the intermediate Jacobian and prove that in some cases it is a semi-abelian variety.
 In this section, we follow Notation \ref{tor33} and notations in \S \ref{subsec1}.
 
 \begin{note}
 Denote by
 $i_1: \mc{G}_0 \cap \mc{G}_1 \hookrightarrow P_0, i_2: \mc{G}_0 \cap \mc{G}_1 \hookrightarrow \mc{G}_1
  \mbox{ and } i_3:\mc{G}_0 \cap \mc{G}_1 \hookrightarrow \mc{G}_0$ the natural inclusions.
  Recall, the kernel of the Gysin morphism from $\mc{G}_0 \cap \mc{G}_1$ to $\mc{G}_1$ and $P_0$:
 \end{note}

 \begin{prop}\label{ner15}
  The kernel of the Gysin morphism $(i_{1,*},i_{2,*})$ is given by 
  \[\ker((i_{1,*},i_{2,*}):H^{i-2}(\mc{G}_0 \cap \mc{G}_1,\mb{Q}) \to H^i(P_0,\mb{Q}) \oplus H^i(\mc{G}_1,\mb{Q})) \cong H^{i-4}(M_{\widetilde{X}_0}(2,\widetilde{\mc{L}}_0)).\]
   \end{prop}
   
   \begin{proof}
    See \cite[Proposition $4.1$]{mumf} for a proof.
    \end{proof}

  Using the definition of intermediate Jacobian in the smooth, projective case, we define generalized 
  intermediate Jacobian of the singular variety $\mc{G}_{X_0}(2,\mc{L}_0)$. 
  
 \begin{defi}
   Define the $i$-th \emph{generalized intermediate Jacobian} of $\mc{G}_{X_0}(2,\mc{L}_0)$ as 
  \[J^i(\mc{G}_{X_0}(2,\mc{L}_0)):=\frac{H^{2i-1}(\mc{G}_{X_0}(2,\mc{L}_0),\mb{C})}{F^i H^{2i-1}(\mc{G}_{X_0}(2,\mc{L}_0),\mb{C})+H^{2i-1}(\mc{G}_{X_0}(2,\mc{L}_0),\mb{Z})}.\]
  \end{defi}
  
  We show that the generalized 
  intermediate Jacobian is a fibration by a product of abelian varieties over the central fiber of the associated
  N\'{e}ron model.
 
 \begin{thm}\label{ner16}
  The specialization morphism $\mr{sp}_{2i-1}$ (as in Proposition \ref{ntor03}) from $H^{2i-1}(\mc{G}_{X_0}(2,\mc{L}_0))$ to $H^{2i-1}(\mc{G}(2,\mc{L})_\infty)$  induces a 
  surjective morphism \[\tau: J^i(\mc{G}_{X_0}(2,\mc{L}_0)) \to \left(\ov{\mbf{J}}^i_{\mc{G}(2, \mc{L})}\right)_0\]
  with kernel isomorphic to $J^{i-1}(M_{\widetilde{X}_0}(2,\widetilde{\mc{L}}_0)) \times J^{i-2}(M_{\widetilde{X}_0}(2,\widetilde{\mc{L}}_0)) \times J^{i-3}(M_{\widetilde{X}_0}(2,\widetilde{\mc{L}}_0))$.
 \end{thm}

 \begin{proof}
  The surjectivity of $\tau$ follows from definition (see \eqref{t02}). We now prove the statement on the kernel of $\tau$.
  There exist closed subschemes $Z \subset P_0$ and $Z' \subset \mc{G}_0$ 
  such that $P_0 \backslash Z \cong \mc{G}_0 \backslash Z'$ (see \S \ref{subsec1}).
  Using \cite{gies} (see also \cite[P. $27$]{tha} or \cite[Remark $6.5$(c), Theorem $6.2$]{sesh4}), one can observe that 
  $Z \cap \Ima(i_1)=\emptyset=Z' \cap \Ima(i_3)$ 
  and there exists a smooth, projective variety $W$ along with proper, birational morphisms $\tau_1:W \to P_0$ and $\tau_2:W \to \mc{G}_0$ such that 
  \[W \backslash \tau_1^{-1}(Z) \cong P_0 \backslash Z \cong \mc{G}_0 \backslash Z' \cong W \backslash \tau_2^{-1}(Z').\]
 Therefore, there exists a natural closed immersion $l:\mc{G}_0 \cap \mc{G}_1 \to W$ such that $i_1=\tau_1 \circ l$ and $i_3=\tau_2 \circ l$.
 We claim that given any $\xi \in H^{k-2}(\mc{G}_0 \cap \mc{G}_1,\mb{Q})$, we have $\tau_1^* \circ i_{1,*}(\xi)=l_*(\xi)=\tau_2^* \circ i_{3,*}(\xi)$.
  Indeed, since $\Ima(i_1)$ (resp. $\Ima(i_3)$) does not intersect $Z$ (resp. $Z'$), the pullback of $l_*(\xi)$ to $\tau_1^{-1}(Z)$ and $\tau_2^{-1}(Z')$ vanish.
  Using the (relative) cohomology exact sequence (\cite[Proposition $5.54$]{pet}), we conclude that there exists $\beta_1 \in H^k(P_0)$ and $\beta_2 \in H^k(\mc{G}_0)$
  such that $\tau_1^*(\beta_1)=l_*(\xi)=\tau_2^*(\beta_2)$. Applying $\tau_{1,*}$ and $\tau_{2,*}$ to the two equalities respectively and using \cite[Proposition $B.27$]{pet}, we get 
  \[\beta_1=\tau_{1,*}\tau_1^*(\beta_1)=\tau_{1,*}l_*(\xi)=i_{1,*}(\xi) \mbox{ and } \beta_2=\tau_{2,*}\tau_2^*(\beta_2)=\tau_{2,*}l_*(\xi)=i_{3,*}(\xi).\]
  In other words, $\tau_1^* \circ i_{1,*}(\xi)=l_*(\xi)=\tau_2^* \circ i_{3,*}(\xi)$. This proves the claim.
  Since Gysin morphisms are morphisms of mixed Hodge structures and $H^{k-2}(\mc{G}_0 \cap \mc{G}_1,\mb{Q})$ is a pure Hodge structure, 
  we have $i_{1,*}(\xi) \in \mr{Gr}^W_k H^k(P_0,\mb{Q})$
  and $i_{3,*}(\xi) \in \mr{Gr}^W_k H^k(\mc{G}_0,\mb{Q})$. Using \cite[Theorem $5.41$]{pet}, we then conclude that $i_{1,*}(\xi)=0$ (resp. $i_{3,*}(\xi)=0$) if and only if $l_*(\xi)=0$. In other words, 
  $\ker(i_{1,*}) \cong \ker(i_{3,*})$. 
  Using Proposition \ref{ner15}, \eqref{ntor02} becomes the following exact sequence of MHS:
   \[0 \to H^{2i-5}(M_{\widetilde{X}_0}(2,\widetilde{\mc{L}}_0),\mb{Q})(-2) \to H^{2i-3}(\mc{G}_0 \cap \mc{G}_1,\mb{Q})(-1) \xrightarrow{f_{2i-1}}
    H^{2i-1}(\mc{G}_{X_0}(2,\mc{L}_0), \mb{Q}) \xrightarrow{\mr{sp}_{2i-1}}\]\[ \xrightarrow{\mr{sp}_{2i-1}} H^{2i-1}(\mc{G}(2,\mc{L})_\infty,\mb{Q}).
   \]
   Recall, $\mc{G}_0 \cap \mc{G}_1$ is a $\mb{P}^1 \times \mb{P}^1$-bundle over $M_{\widetilde{X}_0}(2,\widetilde{\mc{L}}_0)$.
   Denote by \[\rho_1: \mc{G}_0 \cap \mc{G}_1 \to M_{\widetilde{X}_0}(2,\widetilde{\mc{L}}_0)\] the natural projection.
   By the Deligne-Blanchard theorem \cite{delibla} (the Leray spectral sequence
degenerates at $E_2$ for smooth families), we have $H^{k}(\mc{G}_0 \cap \mc{G}_1, \mb{Q}) \cong  \oplus_{_j} H^{k-j}(R^j \rho_{1,*}\mb{Q})$ for all $k \ge 0$.
Since $M_{\widetilde{X}_0}(2,\widetilde{\mc{L}}_0)$ is smooth and simply connected, the local system
$R^j \rho_{1,*} \mb{Q}$ is trivial.
Therefore, for any $y \in M_{\widetilde{X}_0}(2,\widetilde{\mc{L}}_0)$, the natural morphism
\[H^{k}(\mc{G}_0 \cap \mc{G}_1, \mb{Q}) \twoheadrightarrow H^0(R^{k} \rho_{1,*} \mb{Q}) \to H^{k}((\mc{G}_0 \cap \mc{G}_1)_y,\mb{Q})\]
is surjective for all $k \ge 0$.
Then, by the Leray-Hirsch theorem, we have 
\[H^{2i-3}(\mc{G}_0 \cap \mc{G}_1, \mb{Q}) \cong \bigoplus\limits_j H^{2i-3-j}(M_{\widetilde{X}_0}(2,\widetilde{\mc{L}}_0), \mb{Q}) \otimes H^j(\mb{P}^1 \times \mb{P}^1, \mb{Q}).\]
Recall, $H^0(\mb{P}^1 \times \mb{P}^1, \mb{Q}) \cong \mb{Q} \cong H^4(\mb{P}^1 \times \mb{P}^1, \mb{Q})$ and $H^2(\mb{P}^1 \times \mb{P}^1, \mb{Q}) \cong 
\mb{Q} \oplus \mb{Q}$. Furthermore, $H^i(\mb{P}^1 \times \mb{P}^1, \mb{Q})=0$ for $i$ odd and $i>4$.
It is then easy to check that $\ker(\mr{sp}_{2i-1})$ is isomorphic as a pure Hodge structure to 
\[H^{2i-3}(M_{\widetilde{X}_0}(2,\widetilde{\mc{L}}_0),\mb{Q})(-1) \oplus H^{2i-5}(M_{\widetilde{X}_0}(2,\widetilde{\mc{L}}_0),\mb{Q})(-2) \oplus H^{2i-7}(M_{\widetilde{X}_0}(2,\widetilde{\mc{L}}_0),\mb{Q})(-3).\]
 This implies $\ker(\tau) \cong J^{i-1}(M_{\widetilde{X}_0}(2,\widetilde{\mc{L}}_0)) \times J^{i-2}(M_{\widetilde{X}_0}(2,\widetilde{\mc{L}}_0)) \times J^{i-3}(M_{\widetilde{X}_0}(2,\widetilde{\mc{L}}_0))$.
 This proves the theorem.
 \end{proof}

 \begin{cor}\label{ntor18}
  The generalized intermediate Jacobian $J^2(\mc{G}_{X_0}(2,\mc{L}_0))$ is isomorphic to the central fiber $\left(\ov{\mbf{J}}^2_{\mc{G}(2, \mc{L})}\right)_0$
  of the N\'{e}ron model. In particular, $J^2(\mc{G}_{X_0}(2,\mc{L}_0))$ is a semi-abelian variety.
 \end{cor}
 
 \begin{proof}
  Notations as in Theorem \ref{ner16}. In this case, $\ker(\tau)=0$. The corollary then follows immediately from Theorem \ref{ner16} and Corollary \ref{corsemiab}.
 \end{proof}

\end{document}